\documentclass[12pt]{amsart}
\usepackage{fullpage,tikz-cd, amsmath, amssymb, mathrsfs}
\usepackage[alphabetic]{amsrefs}
\usetikzlibrary{arrows}
\title{Platonic and alternating 2-groups}
\author{Narthana Epa}
\author{Nora Ganter}
\thanks{Ganter was supported by an Australian Research
  Fellowship and ARC Discovery Grant DP1095815.}
\date {\today}

\theoremstyle{plain}

\newtheorem{Thm}{Theorem}[section] 
\newtheorem {Lem}[Thm] {Lemma}
\newtheorem {Prop}[Thm] {Proposition}
\newtheorem {Cor}[Thm] {Corollary}

\theoremstyle {definition}
\newtheorem {Def}[Thm] {Definition}

\theoremstyle{remark} 
 \newtheorem {Rem}[Thm] {Remark}
 \newtheorem {Exa}[Thm] {Example}

\newcommand{\mC}{{\mathcal C}}
\newcommand{\mE}{{\mathcal E}}
\newcommand{\mG}{{\mathcal G}}

\newcommand{\mA}{{\mathcal A}}

\newcommand{\mU}{{\mathcal U}}

\newcommand{\bbC}{{\mathbb C}}
\newcommand{\bbH}{{\mathbb H}}
\newcommand{\bbQ}{{\mathbb Q}}
\newcommand{\bbR}{{\mathbb R}}
\newcommand{\bbS}{{\mathbb S}}
\newcommand{\bbZ}{{\mathbb Z}}

\newcommand{\mg}{\boldsymbol\mu_{|G|}}
\newcommand{\mn}{\boldsymbol\mu_n}
\newcommand{\tensor}{\otimes}

\newcommand {\tlongmap}[3] {{
$$
    \begin{tikzpicture}
      \node at (0,0) [name=A, anchor=east] {\ensuremath{#1\negmedspace
        :#2}};
      \node at (1,0) [name=B, anchor=west] {\ensuremath{#3}};
      \draw[->] (A) -- (B);
    \end{tikzpicture}
$$
  }}

\newcommand {\teqnarray}[5] {{
  \begin{center}
    \begin{tikzpicture}
      \node at (0,0) [name=A, anchor=east] {\ensuremath{#1\negmedspace
        :#2}};
      \node at (1,0) [name=B, anchor=west] {\ensuremath{#3}};

      \node at (0,-0.8) [name=C, anchor=east] {\ensuremath{#4}};
      \node at (1,-0.8) [name=D, anchor=west] {\ensuremath{#5}};

      \draw[->] (A) -- (B);
      \draw[|->] (C) -- (D);
    \end{tikzpicture}
  \end{center}
  }\noindent}

\begin{document}
\thispagestyle{empty}
\begin{abstract}
  We recall Schur's work on universal central extensions and develop
  the analogous theory for categorical extensions of groups. We prove
  that the String 2-groups are universal in this sense and study in
  detail their restrictions to the finite subgroups of the Spin
  groups. Of particular interest are subgroups of the 3-sphere
  $Spin(3)$, as well as the
  spin double covers of the alternating groups, whose categorical
  extensions turn out to be governed by the stable 3-stem $\pi_3(\bbS^0)$.
\end{abstract}
\maketitle
\section{Introduction}
By a {\em categorical group} or {\em 2-group}, we mean a small monoidal
groupoid $(\mG,\bullet,1)$ with 
weakly invertible objects. 
We will think of such a $\mG$ as a {\em categorical extension} 
\begin{center}
  \begin{tikzpicture}    
    \node at (0.5,0) [name=A] {$1/\!\!/A$};    
    \node at (2.8,0) [name=B] {$\mG$};    
    \node at (5,0) [name=C] {$G$};    
    \draw[->] (A) -- (B);
    \draw[->] (B) -- (C);
  \end{tikzpicture}
\end{center}  
where $G=\pi_0\mG$ is the group of isomorphism classes of objects in $\mG$ and
the abelian group
$$
A\,=\, \pi_1\mG\,=\,aut_\mG(1)
$$
  is the center of $\mG$.
  The purpose of this note is to study two families of finite categorical groups,
  sitting inside the Lie 2-groups $String(n)$. First, we discuss the platonic 2-groups, 
\begin{center}
  \includegraphics[scale=1.2]{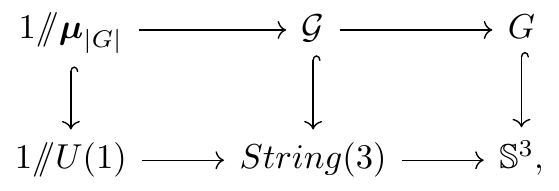}
\end{center}

which are categorical extensions of the finite subgroups of the three sphere and have as center a cyclic group of order $|G|$.
The platonic 2-groups are of interest, because the finite subgroups of
$$\bbS^3\,=\,SU(2)\,=\,Spin(3)$$ are the protagonists of the McKay
correspondence. Their list
consists of the cyclic and the binary dihedral groups,
plus the three exceptional cases: the binary
tetrahedal group $2T\cong\widetilde A_4$, the binary octahedral group
$2O\cong\widetilde S_4$, and the binary icosahedral
group $2I\cong\widetilde A_5$. The fact that there are canonical categorical 
extensions of all these groups suggests a categorical aspect of
McKay correspondence that seems worth exploring.

The second family of examples consists of the {\em alternating
  2-groups} $\mA_n$, which are related to the stable homotopy groups of spheres by the tower
\begin{center}
  \includegraphics[scale=1.2]{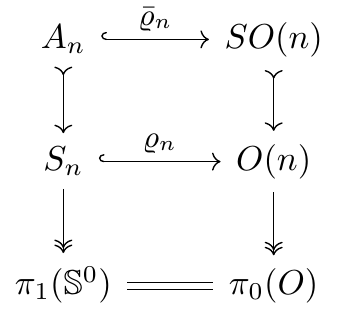}
  \quad\quad
  \includegraphics[scale=1.2]{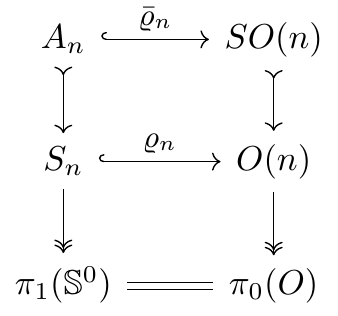}
    \quad\quad
  \includegraphics[scale=1.2]{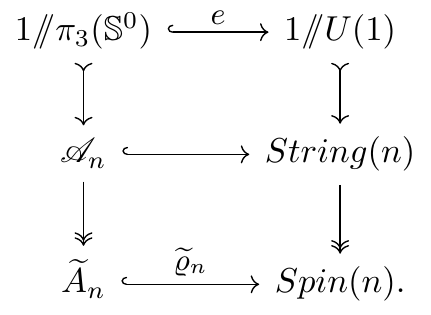}  
\end{center}

Here $$O=colim O(n)$$ is the infinite orthogonal group, and
the homotopy groups turning up are
$$
  \pi_1(\bbS^0)\,\cong\,\boldsymbol\mu_2, \quad\quad
  \pi_2(\bbS^0)\,\cong\,\boldsymbol\mu_2, \quad\quad
  \pi_3(\bbS^0)\,\cong\,\boldsymbol\mu_{24},\quad\quad\text{ and }\quad\quad
  B\pi_3(O)\,=\, U(1). 
$$
The homomorphism
$\widetilde\varrho_n$ is the permutation representation, and
$$%
  \includegraphics[scale=1.2]{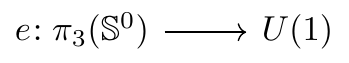}
$$ 
  is the Adams $e$-invariant.
  In the philosophy of \cite{Kapranov15}, the stable 1-stem yields the sign governing super-symmetry, while the stable 2-stem provides the sign governing categorified supersymmetry. It was Kapranov's question about a conceptual description of the stable 3-stem in this context that motivated our work.
  For $n$ sufficiently large, the alternating 2-groups turn out to be universal in an appropriate sense.
  A consequence of this is the following result.

\begin{Thm}\label{thm:24}
  The restriction of $String(n)$ to $\widetilde A_n$ has
  exact order $24$ for all $n\geq4$.
  The restriction of $String(3)$ to a finite subgroup $G\subseteq\bbS^3$ has
  exact order $|G|$.
\end{Thm}
It would be interesting to have a direct proof of Theorem
\ref{thm:24}, using any of the known constructions of the String
2-groups. Further, one can think of $\pi_3(\bbS^0)$
as framed bordism group, 
generated by the three sphere in its invariant
framing and use a $K3$-surface with little holes cut out as a null-bordism
of $24[\bbS^3]$, suggesting a potential connection with the
categorical groups turning up in Mathieu Moonshine.

%
\subsection{Acknowledgments}
This paper is based on the first author's Honours thesis.
It is a pleasure to thank Mikhail Kapranov for generously sharing his ideas on the subject and the
American Institute of Mathematics under whose hospitality the idea for the project was formed. The first author would like to thank Tobias Dyckerhoff for helpful input. The second author would like to thank
Mamuka Jibladze and Gerd Laures for helpful correspondence and Michael Hopkins for helpful conversations.
We would also like to thank John Baez for his feedback on an earlier draft and for suggesting the very romantic name `Platonic 2-groups'.
\section{Extensions and group cohomolgy}
\label{sec:background}
Let $G$ be a group, and write 
$$
  H_*(G)\,=\,H_*^{gp}(G,\bbZ)
$$
for the group homology of $G$ with coefficients in the trivial
$G$-module $\bbZ$. Then
$$
  H_1(G)\, \cong\, G^{ab}
$$
is the abelianization of $G$, and $H_2(G)$ is the {\em Schur
  multiplier} of $G$. We will refer to $H_3(G)$ as the {\em
  categorical Schur multiplier} of $G$.
Recall that a group $G$ is called {\em
  perfect} if its abelianization is trivial and that 
a perfect group is called 
{\em superperfect} if its Schur multiplier also vanishes. The smallest
non-trival example of a superperfect group is the binary icosahedral
group, whose categorical Schur multiplier is 
\[
  H_3(2I)\,\cong\,\boldsymbol\mu_{120},
\]
see \cite{Hausmann78}. A list of the categorical Schur multipliers of
some superperfect groups exists as HAP library.
\begin{Def}
  A central extension
  \begin{center}
    \begin{tikzpicture}
      \node at (-0.2,0) [name=A] {$A_{uni}$};
      \node at (2,0) [name=B] {$\widetilde G_{uni}$};
      \node at (4,0) [name=C] {$G$};
      \draw[->] (A) -- (B);
      \draw[->] (B) -- (C);
    \end{tikzpicture}
  \end{center}
  of finite dimensional Lie groups is called a {\em Schur cover} of
  $G$, if it is universal in the following sense: for
  any finite dimensional central extension
  \begin{center}
    \begin{tikzpicture}
      \node at (0,0) [name=A] {$A$};
      \node at (2,0) [name=B] {$\widetilde G$};
      \node at (4,0) [name=C] {$G$};
      \draw[->] (A) -- (B);
      \draw[->] (B) -- (C);
    \end{tikzpicture}
  \end{center}
  of $G$ there exists a unique map of central extensions
  \begin{center}
    \begin{tikzpicture}
      \node at (-0.2,0) [name=A] {$A_{uni}$};
      \node at (2,0) [name=B] {$\widetilde G_{uni}$};
      \node at (4,0) [name=C] {$G$};
      \draw[->] (A) -- (B);
      \draw[->] (B) -- (C);

      \node at (-0.2,-1.8) [name=A1] {$A$};
      \node at (2,-1.8) [name=B1] {$\widetilde G$};
      \node at (4,-1.8) [name=C1] {$G$.};
      \draw[->] (A1) -- (B1);
      \draw[->] (B1) -- (C1);

      \draw[->] (A) -- (A1);
      \draw[->] (B) -- (B1);
      \draw[double equal sign distance] (C) -- (C1);
    \end{tikzpicture}
  \end{center}
  
\end{Def}
If it exists, the Schur cover of $G$ is unique up to
unique isomorphism. 
We recall two classical results about
Schur covers.
\begin{Thm}[Schur 1904] 
  \label{thm:Schur_cover}
  Let $G$ be a perfect discrete group. Then $G$ possesses
  a Schur cover, whose central subgroup is the Schur multiplier 
  $$A_{uni}\,=\,H_2(G).$$  
\end{Thm}
\begin{Lem}[Second Whitehead Lemma]\label{lem:Whitehead}
  Let $G$ be a semisimple, compact and connected Lie group. Then the
  universal covering group of $G$ is a Schur cover.
  Its central subgroup is the fundamental group $$A_{uni}\,=\,\pi_1(G).$$
\end{Lem}
To ephasize the analogy between these two statements, let $BG$
be the classifying space of $G$.
Then we have
$$
  \pi_i(BG) \,=\, \pi_{i-1}(G).
$$
So, the Lie group $G$ is connected if and
only if $\pi_1(BG)$ is trivial. In this case, we have
$$
  H_1(BG;\bbZ) \;=\; 0
$$
and
$$
  H_2(BG;\bbZ) \;=\; \pi_2(BG)\;\cong\;\pi_1(G).
$$
The goal of this section is to develop the theory of
Schur covers in the context of categorical central extensions.
Let $G$ be a finite dimensional Lie group, and write $\mE\!xt(G)$ for the bicategory of finite dimensional central Lie 2-group extensions
  \begin{center}
    \begin{tikzpicture}
      \node at (0,0) [name=A] {$1/\!\!/A$};
      \node at (2.2,0) [name=B] {$\mG$};
      \node at (4.2,0) [name=C] {$G$,};
      \draw[->] (A) -- (B);
      \draw[->] (B) -- (C);
    \end{tikzpicture}
  \end{center}
as in \cite{Schommer-Pries11}.
{\em Central} in this context means that the conjugation action of $G$ on $A$ is trivial,
{\em Lie} means that $\mG$ is a finite dimensional Lie groupoid and that the additional data (tensor multiplication, associator, 
etc.) are required to be locally continuous and smooth in an appropriate
sense.
\begin{Def}
  A {\em categorical Schur} cover of $G$ is an
  initial object in $\mE\!xt(G)$. 
\end{Def}
Explicitly,
  \begin{center}
    \begin{tikzpicture}
      \node at (-0.3,0) [name=A] {$1/\!\!/A_{uni}$};
      \node at (2.2,0) [name=B] {$\mG_{uni}$};
      \node at (4.3,0) [name=C] {$G$};
      \draw[->] (A) -- (B);
      \draw[->] (B) -- (C);
    \end{tikzpicture}
  \end{center}
  is a categorical Schur cover of $G$ if for any other finite dimensional
central Lie 2-group extension as above
there exists a 1-morphism
  \begin{center}
    \begin{tikzpicture}
      \node at (2.5,0) [name=B] {$1/\!\!/A_{uni}$};
      \node at (5.5,0) [name=C] {$\mG_{uni}$};
      \node at (8.5,0) [name=D] {$G$};
    
      \draw[->] (B) -- (C);
      \draw[->] (C) -- (D);

      \node at (2.5,-1.8) [name=B2] {$1/\!\!/ A$};
      \node at (5.5,-1.8) [name=C2] {$\mG$};
      \node at (8.5,-1.8) [name=D2] {$G$};
    
      \draw[->] (B2) -- (C2);
      \draw[->] (C2) -- (D2);
 
      \draw[->] (B) -- node [midway, left] {} (B2);
      \draw[->] (C) -- (C2);
      \draw[double equal sign distance] (D) -- node [midway, right] {} (D2);
    \end{tikzpicture}
  \end{center}
  
in $\mE\!xt(G)$, which is unique up to unique 2-isomorphism.
If it exists, the categorical Schur cover of $G$
is unique up to equivalence, which in turn is unique up to unique
isomorphism.
The goal of this section is to prove the following result.
\begin{Thm}\label{thm:Categorical_Schur_cover}
  \begin{enumerate}
  \item Let $G$ be a superperfect discrete group. Then $G$ possesses a
    categorical Schur cover, whose center is
    $$A_{uni}=H_3(G).$$
  \item Let $G$ be a simply connected compact Lie group, and let $s$
    be the number of simple factors of $G$. Then $G$ possesses a
    categorical Schur cover, whose center is $$A_{uni}=U(1)^s.$$
  \end{enumerate}
\end{Thm}
Note that simply connected compact Lie groups are automatically
semi-simple \cite[Thm.5.29]{MimuraToda91}, so that the statement in
(2) makes sense.
  Let
  \begin{center}
  \begin{tikzpicture}
    \node at (2,0) [name=b] {$C^0$};
    \node at (4,0) [name=c] {$C^1$};
    \node at (6,0) [name=d] {$\dots$};
    \node at (8,0) [name=e] {$C^{n-1}$};
    \node at (10,0) [name=f] {$Z^n$};

    \draw[->] (b) -- node [midway, above] {$d$} (c);
    \draw[->] (c) -- node [midway, above] {$d$} (d);
    \draw[->] (d) -- node [midway, above] {$d$} (e);
    \draw[->>] (e) -- node [midway, above] {$d$} (f);
  \end{tikzpicture}
  \end{center}
  be a cochain complex of abelian groups. Recall that the {\em Dold-Kan
  $n$-groupoid} of $C^\bullet$ is the
  (strictly symmetric monoidal) strict $n$-groupoid
  with objects $Z^n$ and arrows
  \begin{eqnarray*}
    1Hom (\alpha,\beta) & = & \{\gamma\in C^{n-1}\mid d\gamma =
                             \beta-\alpha\}\\
    2Hom (\gamma,\delta) & = & \{\epsilon\in C^{n-2}\mid d\epsilon =
                              \delta-\gamma\}\\
    & \dots &
  \end{eqnarray*}
  Composition of arrows is given by addition, and so is the monoidal
  structure.   
The following theorem summarizes results of Schur (1911),
Singh (1976),
Schommer-Pries \cite[Thm.99]{Schommer-Pries11}, Wagemann and Wockel
\cite{WagemannWockel15}, Schreiber \cite{Schreiber13}.
\begin{Thm}\label{thm:DoldKan}
  Let $G$ and $A$ be finite dimensional Lie groups with $A$ abelian.
  Let $C^\bullet_{gp}(G;A)$ be the cochain complex of locally
  continuous group cocycles on $G$ with values in the trivial
  $G$-module $A$ as in
  \cite[Def.I.1]{WagemannWockel15}.
  Then the Dold-Kan groupoid of
  \begin{center}
    \begin{tikzpicture}
      \node at (2,0) [name=b] {$C^1_{gp}(G;A)$};
      \node at (5.5,0) [name=c] {$Z^2_{gp}(G;A)$};

      \draw[->>] (b) -- node [midway, above] {$d$} (c);
    \end{tikzpicture}
  \end{center}
  is equivalent to the symmetric monoidal category of central extensions of the form
  \begin{center}
    \begin{tikzpicture}
      \node at (0,0) [name=A] {$A$};
      \node at (2,0) [name=B] {$\widetilde G$};
      \node at (4,0) [name=C] {$G$.};
      \draw[->] (A) -- (B);
      \draw[->] (B) -- (C);
    \end{tikzpicture}
  \end{center}
  The Dold-Kan 2-groupoid of
  \begin{center}
    \begin{tikzpicture}
      \node at (2,0) [name=b] {$C^1_{gp}(G;A)$};
      \node at (5.5,0) [name=e] {$C^2_{gp}(G;A)$};
      \node at (9,0) [name=c] {$Z^3_{gp}(G;A)$};

      \draw[->] (b) -- node [midway, above] {$d$} (e);
      \draw[->>] (e) -- node [midway, above] {$d$} (c);
    \end{tikzpicture}
  \end{center}
  is equivalent to the symmetric monoidal bicategory of categorical central extensions of
  the form
  \begin{center}
    \begin{tikzpicture}
      \node at (0,0) [name=A] {$1/\!\!/A$};
      \node at (2,0) [name=B] {$\mG$};
      \node at (4,0) [name=C] {$G$.};
      \draw[->] (A) -- (B);

      \draw[->] (B) -- (C);
    \end{tikzpicture}
  \end{center}
\end{Thm}
To be specific, let $\alpha$ be
an $A$-valued 3-cocycle on $G$. Then we have the (skeletal) groupoid
\begin{center}
  \begin{tikzpicture}
     \node at (-.1,0) [name=A] {$\mG_\alpha$\,};
     \node at (1.8,0) [name=B] {\,
     $     \left(\!
       \raisebox{-.75cm}{\begin{tikzpicture}
         \node at (0,.6) [name=C] {$G\times A$};
         \node at (0,-.6) [name=D] {$G$};
         \draw[->] ([xshift=-6pt]C.south) -- ([xshift=-6pt]D.north);
         \draw[->] ([xshift=6pt]C.south) -- ([xshift=6pt]D.north);
       \end{tikzpicture}}\!
     \right)$

     };    
     \draw[double equal sign distance] (A) -- (B);
  \end{tikzpicture}
\end{center}

The group multiplications give a monoidal structure on $\mG_\alpha$
with the associator encoded in $\alpha$, and the
unit maps trivial.
This is the categorical group associated to alpha by the equivalence
in the theorem.
\begin{Cor}[{\cite{Brown94},\cite{Schommer-Pries11}}]\label{Cor:vanishing}
  Assume we are given finite dimensional Lie groups $G$ and $A$ with
  $A$ abelian. Then the following hold.
  \begin{enumerate}
  \setlength{\itemsep}{3pt}
  \item Classes in $H^2_{gp}(G;A)$ are in one to one correspondence
    with isomorphism classes of central extensions of $G$ by $A$.
  \item If we have
    $$
      H^1_{gp}(G;A) \cong 0,
    $$
    then a degree 2 class as in (1) determines the corresponinding
    central extension uniquely up to unique isomorphism.
  \item Classes in $H^3_{gp}(G;A)$ are in one to one correspondence
    with equivalence classes of categorical central extensions of $G$
    with center $A$.
  \item If we have
    $$
      H^i_{gp}(G;A) \cong 0
    $$
    for $i=1$ and $i=2$, then a degree 3 class as in (3) determines the corresponinding
    categorical central extension uniquely up to equivalence, which in
    turn is unique up to unique isomorphism.
  \end{enumerate}
\end{Cor}
Let $G$ be a discrete group. 
Write $BG$ for the classifying space (geometric realization of the nerve) of $G$. 
Then we have isomorphisms
\begin{eqnarray*}
\setlength{\itemsep}{3pt}
  H_i^{gp}(G) & \cong &  H_i(BG), \\ 
  H^i_{gp}(G;A) & \cong &  H^i(BG;A), 
\end{eqnarray*}
where on the right-hand side we have singular (co)homology.
The universal coefficient theorem gives the short exact sequence
\begin{center}
  \begin{tikzpicture}[scale=1.8]
    \node at (-0.1,0) [name=A] {$0$};
    \node at (1.7,0) [name=B] {$Ext\left(H_{i-1}(G),A\right)$};
    \node at (4,0) [name=C] {$H^i(G;A)$};
    \node at (6.2,0) [name=D] {$Hom\left(H_{i}(G),A\right)$};
    \node at (8,0) [name=E] {$0$,};

    \draw[->] (A) -- (B);
    \draw[->] (D) -- (E);
    \draw[->] (B) -- (C);
    \draw[->] (C) -- (D);
  \end{tikzpicture}
\end{center}
natural in $A$.
\begin{proof}[{Proof of Theorems \ref{thm:Schur_cover} and \ref{thm:Categorical_Schur_cover}
  (1)}]
  Assume that $G$ is perfect, and let $A$ be an abelian group, viewed
  as trivial $G$-module. Then the universal coefficient theorem with $i=1$ implies 
  \[
    0\,=\, H^1(G;A) \, = \,Z^1(G,A).
  \]
  By Theorem \ref{thm:DoldKan}, it follows that the groupoid of
  central extensions of $G$ by $A$ is discrete in the sense that there
  are no non-identity automorphisms. Using the universal coefficient
  theorem with $i=2$, it follows that the isomorphism classes of said
  groupoid are parametrised by $Hom(H_2(G),A)$. Now allow $A$ to
  vary. Then we obtain an equivalence from the category of central
  extensions of $G$ to the under category $H_2(G)\downarrow\mA b$,
  sending a central extension to the homomorphism classifying it. In
  particular, there is a universal central extension, which is
  characterised uniquely, up to unique isomorphism, by the fact that it
  is classified by $id_{H_2(G)}$. The proof for Theorem \ref{thm:Categorical_Schur_cover}
  (1) is analogous.
\end{proof}
We also have the following corollary of the universal coefficient
theorem (using injectivity of the circle group).
\begin{Cor}
  If $G$ is finite,
  we have a non-canonical isomorphism
  \begin{center}
    \begin{tikzpicture}[scale=1.8]
      \node at (4.5,0) [name=C] {$H^i(G;U(1))$};
      \node at (7,0) [name=D]
      {$\widehat{H_{i}(G)}\,\cong\,H_{i}(G)$.};
    
      \draw[->] (C) -- node [midway, above] {$\cong$} (D);
    \end{tikzpicture}
  \end{center}
  In particular, the (categorical) Schur multiplier of a finite group
  $G$ classifies (categorical) central extensions of
  $G$ by the circle group.
\end{Cor}
These categorical central extensions by the circle group
are of interest for the theory of projective
2-representations  \cite{GanterUsher14}, just like central group extensions by the circle group are of interest for the theory of projective
representations.
\begin{Def}
  Let $\mG$ be a categorical group with center $A$. Let 
  $\phi\negmedspace:A\to B$ be a homomorphism to another abelian
  group. Then the {\em categorical group with center $B$ associated to
    $\mG$} (via $\phi$) is the groupoid 
  $\mG[\phi]$ 
  with objects identical to those of $\mG$ and arrows the 
  pairs $(f,b)$ with $f$ an arrow of $\mG$ and $b\in B$, modulo the
  equivalence relation 
  $$
    (f\bullet a,b) \,\sim\,(f,\phi(a)+b), \quad\quad a\in A.
  $$
  The multiplication data are inherited from $\mG$.
\end{Def}
Given an abelian group $H$, we will write $H\!\downarrow\!\mA b$ for
the category of abelian groups under $H$.
\begin{Thm} 
  \begin{enumerate}
  \item 
  For a perfect group $G$ with Schur cover $\widetilde G_{uni}$,
  the functor
  \begin{center}
    \begin{tikzpicture}
      \node at (0,0) [name=A, anchor=east] {\ensuremath{\widetilde
          G_{uni}[-]:H_2(G)\!\downarrow\mA b}\,}; 
      \node at (1.3,0) [name=B, anchor=west] {\,\ensuremath{Ext (G)}};
      \draw[->] (A) -- node [midway, above] {} (B);
    \end{tikzpicture}
  \end{center}
  sending the homomorphism
  $$
    \phi:{H_2(G)}\longrightarrow A
  $$
  to the balanced product
  $$
    A\times_{H_2(G)}\widetilde G_{uni}
  $$
  is an equivalence of categories.
  \item For a superperfect group $G$ with categorical Schur cover $\mG_{uni}$, 
  the 
  functor
  \begin{center}
    \begin{tikzpicture}
      \node at (0,0) [name=A, anchor=east] {\ensuremath{\mG_{uni}[-]:
          H_3(G)\!\downarrow\mA b}\,}; 
      \node at (1.3,0) [name=B, anchor=west] {\,\ensuremath{\mE\! xt (G)}};

      \draw[->] (A) -- node [midway, above] {} (B);
    \end{tikzpicture}
  \end{center}
   is an equivalence of (bi)categories. 
   \end{enumerate}
\end{Thm}
\begin{proof}
  The fisrt part is classical, we prove (2).
  The universal coefficient theorem implies that the bicategory
  $\mE\!xt (G)$ has only identity 2-morphisms and that we have an abstract
  equivalence
  \begin{center}
    \begin{tikzpicture}
      \node at (0,0) [name=A, anchor=east] {\ensuremath{
          H_3(G)\!\downarrow\mA b}\,}; 
      \node at (1.3,0) [name=B, anchor=west] {\,\ensuremath{\mE\! xt (G)}.};
      \draw[->] (A) -- node [midway, above] {$\sim$} (B);
    \end{tikzpicture}
  \end{center}
  The identity map of $H_3(G)$ is an initial object of $H_3(G)\!\downarrow\mA b$.
  Under the isomorphism of the universal coefficient theorem, this corresponds to the class of
  a 3-cocycle $\alpha_{uni}$ with values in
  $H_3(G)$, and for arbitrary $A$, the universal coefficient isomorphism is
  \begin{eqnarray*}
     H^3(G;A) & \cong & Hom(H_3(G), A)
    \\
    \left[\phi_*\,\alpha_{uni}\right] & \longmapsto & \phi,
  \end{eqnarray*}  
  by naturality. If
  $$\mG_{uni}\,=\,\mG_{\alpha_{uni}},$$
  then, by construction,
  $$\mG_{\phi_*\alpha_{uni}}\,\simeq\,\mG_{uni}[\phi].$$ 
  An inverse of the functor $\mG_{uni}[-]$ 
  restricts the unique 1-morphism
  $\mG_{uni}\longrightarrow\mG$ to centers.
\end{proof}
\begin{Def}
  Let $G$ be a discrete group, not necessarily perfect. Assume
  that the Schur multiplier $H_2(G)$ vanishes. Then we still have the cohomology class
  $[\alpha_{uni}]$ and $\mG_{uni}$ as in the above proof. We will
  refer to any Lie 2-group extension equivalent to $\mG_{uni}$ as a 
  {\em weak categorical Schur cover} of
  $G$.
\end{Def}
In the situation of the definition,
\begin{center}
  \begin{tikzpicture}
    \node at (0,0) [name=A, anchor=east] {\ensuremath{\mG_{uni}[-]: H_3(G)\!\downarrow\mA b}\,};
    \node at (1.3,0) [name=B, anchor=west] {\,\ensuremath{\mE\! xt (G)}};
    \draw[->] (A) -- node [midway, above] {} (B);
  \end{tikzpicture}
\end{center}
is still essentially bijective, but may no longer be an equivalence of
bicategories.
Let now $G$ be a simply connected compact Lie group, and let $T$ be a
finite dimensional abelian Lie group with cocharacter lattice
$$
  \check T \; := \; Hom( U(1) , T ). 
$$
\begin{Thm}\label{thm:Lie}
  Let $s$ be the number of simple factors of $G$.
  Then
  $$
    H^1_{gp}(G;T) \;=\; H^2_{gp}(G;T) \;=\;0,
  $$
  and we have an isomorphism
  $$
    H^3_{gp}(G;T) \;\cong \; 
    \check T^{\,s},
  $$
  which is natural in $G$. 
\end{Thm}
\begin{proof}
  We have
  $$%
    \pi_i(BG)\;=\;\pi_{i-1}(G) = 
    \begin{cases}
      0 & i\leq 3 \\
      \bbZ^s & i = 4,
    \end{cases}
  $$
  \cite[Thm.~4.17]{MimuraToda91}. Using Hurewitz and the universal
  coefficient theorem, this implies
  $$%
    H^i_{gp}(G;A)\;=\; H^i(BG;A) \;=\;
    \begin{cases}
      0 & 1\leq i\leq 3\\
      A^s & i=4
    \end{cases}
  $$
  for discrete coefficients $A$. If $T_0$ is the connected component of $0$, then
  the short exact sequence
  \begin{center}
    \begin{tikzpicture}
      \node at (0,0) [name=a, anchor=east] {$T_0$};
      \node at (2,0) [name=b] {$T$};
      \node at (4,0) [name=c, anchor=west] {$T/T_0$,};
      \draw[right hook->] (a) -- (b);
      \draw[->>] (b) -- (c);
    \end{tikzpicture}
  \end{center}
  gives isomorphisms
  $$
    H^i_{gp}(G;T_0) \;\cong\; H^i_{gp}(G;T)
  $$
  for $1\leq i\leq3$. Let $S\subseteq T_0$ be a maximal compact
  subgroup. Then $T_0$ is the product of $S$ with $\bbR^m$ for some
  $m$. By \cite[Thm.~2.8]{Hu52}, the cohomology of a compact Lie group
  with coefficients in $\bbR^m$ vanishes for $i\geq 1$. Hence the short exact sequence
  \begin{center}
    \begin{tikzpicture}
      \node at (0,0) [name=a, anchor=east] {$S$};
      \node at (2,0) [name=b] {$T_0$};
      \node at (4,0) [name=c, anchor=west] {$T_0/S$};
      \draw[right hook->] (a) -- (b);
      \draw[->>] (b) -- (c);
    \end{tikzpicture}
  \end{center}
  gives isomorphisms
  $$
    H^i_{gp}(G;S) \;\cong\; H^i_{gp}(G;T_0).
  $$
  At the same time, we have 
  $$
    \check S\;=\;\check T.
  $$
  We may therefore assume, without loss of generality, that $T$ is a
  compact torus. Finally, the long exact cohomology sequence for 
  \begin{center}
    \begin{tikzpicture}
      \node at (0,0) [name=a, anchor=east] {$\check T$};
      \node at (2,0) [name=b] {$\mathfrak t$};
      \node at (4,0) [name=c, anchor=west] {$T$,};
      \draw[right hook->] (a) -- (b);
      \draw[->>] (b) -- (c);
    \end{tikzpicture}
  \end{center}
  with $\mathfrak t$ the Lie algebra of T, gives isomorphisms
  $$
    H^i_{gp}(G;T) \;\cong\; H^{i+1}(BG;\check T).
  $$  
\end{proof}
The statements of Lemma \ref{lem:Whitehead} and Theorem
\ref{thm:Categorical_Schur_cover} (2) can be derived from this result:
\begin{proof}[Proof of the Second Whitehead Lemma]
  By Weyl's theorem, the simply connected group $\widetilde G$ is
  again compact, and by Theorem \ref{thm:Lie}, it has no non-trivial central
  extension with finite dimensional center.
\end{proof}
\begin{proof}[Proof of Theorem \ref{thm:Categorical_Schur_cover} (2)]
  We have functorial isomorphisms
  \begin{eqnarray*}
    \check T^s &\;=\;& Hom(\bbZ^s,Hom(U(1),T))\\
      &\cong &  Hom(U(1)\tensor\bbZ^s,T) \\
      & = & Hom(U(1)^s,T).
  \end{eqnarray*}
  So, Theorem \ref{thm:Lie} implies that the bicategory $\mE\!xt(G)$
  is equivalent to the category 
  of finite dimensional abelian Lie groups under $U(1)^s$.
\end{proof}
\begin{Exa}[{\cite[Thm.100]{Schommer-Pries11}}]
  For $n=3$ and $n\ge5$, the string extension
  \begin{center}
    \begin{tikzpicture}
      \node at (-0.6,0) [name=A] {$1/\!\!/U(1)$};
      \node at (2.65,0) [name=B] {$String(n)$};
      \node at (6,0) [name=C] {$Spin(n)$};
      \draw[->] (A) -- (B);
      \draw[->] (B) -- (C);
    \end{tikzpicture}
  \end{center}
  is the universal central Lie 2-group extension of the simple and
  simply connected Lie group $Spin(n)$.
\end{Exa}
\section{The cyclic groups}\label{sec:Abelian}
As a warm-up to the platonic and alternating case, we study the categorical extensions of the finite subgroups of the circle group.
The finite cyclic groups have integral homology 
  $$%
    H_i(\mn)\, =\,
    \begin{cases}
      \bbZ & \quad{i=0,}\\
      \mn & \quad{i \text{ odd, and}}\\
      0 & \quad{\text{else.}}
    \end{cases}
  $$
This implies that $\mn$ possesses a weak categorical Schur cover
$\mC_n$, whose centre is $\mn$. Let $\mU(1)^-$ be the categorical extension
of the circle group classified by the standard generator\footnote{In \cite{Ganter14}, we make the convention that the basic categorical extension of the circle group is the 2-group $\mU(1)$ classified by the other generator. These two 2-groups differ by a sign in the action.} of
$$
  H^3_{gp}(U(1);U(1))\,\cong\,H^4(BU(1);\bbZ)\,\cong\,\bbZ.
$$
We will see that there is a 1-morphism of categorical central
extensions
\begin{center}
  \begin{tikzpicture}
    \node at (2.5,0) [name=B] {$1/\!\!/\mn$};
    \node at (5.5,0) [name=C] {$\mC_n$};
    \node at (8.5,0) [name=D] {$\mn$};

    \draw[->] (B) -- (C);
    \draw[->] (C) -- (D);

    \node at (2.5,-1.7) [name=B2] {$1/\!\!/ U(1)$};
    \node at (5.5,-1.7) [name=C2] {$\mU(1)^-$};
    \node at (8.5,-1.7) [name=D2] {$U(1)$,};
    
    \draw[->] (B2) -- (C2);
    \draw[->] (C2) -- (D2);

    \draw[right hook->] (B) -- node [midway, left] {} (B2);
    \draw[right hook->] (C) -- (C2);
    \draw[right hook->] (D) -- node [midway, right] {} (D2);
  \end{tikzpicture}
\end{center}
identifying $\mC_n$ with a sub-categorical group of $\mU(1)^-$.
Let $\bbR$ act on $\bbZ\times U(1)$ by
\begin{align*}
  x\cdot (m,z) & \: := \: \left(m,z\cdot e^{-2\pi imx}\right), 
\end{align*}
and recall that
\begin{center}
  \begin{tikzpicture}
     \node at (-1.2,0) [name=A] {$\mU(1)^-$\,};
     \node at (1.8,0) [name=B] {\,
     $     \left(\!
       \raisebox{-.75cm}{\begin{tikzpicture}
         \node at (0,.6) [name=C] {$\bbR\ltimes(\bbZ\times U(1))$};
         \node at (0,-.6) [name=D] {$\bbR$};
         \draw[->] ([xshift=-6pt]C.south) -- ([xshift=-6pt]D.north);
         \draw[->] ([xshift=6pt]C.south) -- ([xshift=6pt]D.north);
       \end{tikzpicture}}\!
     \right)$
     };    
     \draw[double equal sign distance] (A) -- (B);
  \end{tikzpicture}
\end{center}
is constructed as the strict categorical group
corresponding to the crossed module
  \begin{center}
  \begin{tikzpicture}
    \node at (0,0) [name=A, anchor=east] {$\boldsymbol\upsilon :\bbZ\times U(1)$};
    \node at (1.3,0) [name=B, anchor=west] {$\bbR$};
    \draw[->] ([yshift=-1pt]A.east) -- ([yshift=-1pt]B.west);

    \node at (0,-0.7) [name=A, anchor=east] {$(m,z)$};
    \node at (1.3,-0.7) [name=B, anchor=west] {$m$,};
    \draw[|->] ([yshift=-1pt]A.east) -- ([yshift=-1pt]B.west);
  \end{tikzpicture}
  \end{center}
see \cite{Ganter14}.
In other words, $\mU(1)^-$
has as objects $\bbR$ and as arrows
$$
  \left\{x\xrightarrow{\,\,\,\,z\,\,\,\,}x+m \,\,\middle|\,\,
    x\in\bbR,\,m\in\bbZ,\,\text{ and }z\in U(1)\right\}, 
$$
composing two arrows means multiplying their labels, and the strict
monoidal structure is given by the respective group structures of
objects and arrows.
\begin{Lem}\label{lem:cyclic_Schur}
  The weak Schur cover $\mC_n$ of $\mn$ can be constructed as the
  strict categorical group corresponding to the sub-crossed module
  \begin{center}
    \begin{tikzpicture}[scale=.95]
      \node at (0,0) [name=A] {$\bbZ\times \boldsymbol\mu_n$};
      \node at (3.2,0) [name=B] {$\bbZ\times U(1)$};

      \node at (0,-2) [name=C] {$\frac1n\bbZ$};
      \node at (3.2,-2) [name=D] {$\bbR$.};

      \draw[->] (A) -- node [midway, left] {$\boldsymbol\kappa$} (C);
      \draw[right hook->] (A) -- node [midway, above] {} (B);
      \draw[right hook->] (C) -- node [midway, above] {} (D);
      \draw[->] (B) -- node [midway, right] {$\boldsymbol\upsilon$} (D);
    \end{tikzpicture}
  \end{center}
\end{Lem}
\begin{proof}
The circle group $U(1)$ acts by multiplication on the
 spheres $\bbS^{2k-1}\subset\bbC^k$, and on $$\bbS^\infty \,=\,
 colim_k\bbS^{2k-1}.$$
We have 
$$
\begin{array}{rcccl}
  BU(1) & \;\simeq\; & \bbS^\infty/U(1) & \;=\; & \bbC P^\infty \\[4pt]
  B\boldsymbol\mu_n & \;\simeq\; & \bbS^\infty/\boldsymbol\mu_n & \;=\; & L_n^\infty
\end{array}
$$
(infinite dimensional lens space).
Let
$i\negmedspace : \mn\hookrightarrow U(1)$
be the inclusion map.
Then 
\tlongmap{Bi}{B\boldsymbol\mu_n}{BU(1)}
is identified with the quotient map 
\begin{center}
  \begin{tikzpicture}
    \node at (0,0) [anchor=east, name=a] {$L_n^\infty$};
    \node at (1,0) [anchor=west, name=b] {$\bbC P^\infty$.};
    \draw[->] (a) -- (b);
  \end{tikzpicture}
\end{center}
This is a fibration with fiber
$\bbS^1\cong U(1)/\boldsymbol\mu_n$. Its homology Leray-Serre
spectral sequence has $E^2$-term 
  \begin{center}
    \begin{tikzpicture}[scale=.8]

    \draw[very thin,color=gray!20](-.6,-.6)grid(7.9,2.9);
    \draw[->, very thick](-.6,0)--(8.2,0) node[right] {$\bbC P^\infty$};
    \draw[->, very thick](0,-.6)--(0,3.2) node[above]
    {$\bbS^1$};

    \node at (.5,.5) {$\bbZ$};
    \node at (2.5,.5) [name=b1] {$\bbZ$};
    \node at (4.5,.5) [name=b2] {$\bbZ$};
    \node at (6.5,.5) [name=b3] {$\bbZ$};

    \node at (.5,1.5) [name=a1] {$\bbZ$};
    \node at (2.5,1.5) [name=a2] {$\bbZ$};
    \node at (4.5,1.5) [name=a3] {$\bbZ$};
    \node at (6.5,1.5) {$\bbZ$};

    \draw[red!65!black, thick, ->] (b1) -- node [pos=.37, above]
    {$n\cdot$} (a1);
    \draw[red!65!black, thick, ->] (b2) -- node [pos=.37, above]
    {$n\cdot$} (a2);
    \draw[red!65!black, thick, ->] (b3) -- node [pos=.37, above]
    {$n\cdot$} (a3);
  \end{tikzpicture}
  \end{center}
For cohomology with coefficients in $A$, we obtain an $E_2$-term of the form
  \begin{center}
  \begin{tikzpicture}[scale=.8]

    \draw[very thin,color=gray!20](-.6,-.6)grid(7.9,2.9);
    \draw[->, very thick](-.6,0)--(8.2,0) node[right] {$\bbC P^\infty$};
    \draw[->, very thick](0,-.6)--(0,3.2) node[above]
    {$\bbS^1$};

    \node at (.5,.5) {$A$};
    \node at (2.5,.5) [name=b1] {$A$};
    \node at (4.5,.5) [name=b2] {$A$};
    \node at (6.5,.5) [name=b3] {$A$};

    \node at (.5,1.5) [name=a1] {$A$};
    \node at (2.5,1.5) [name=a2] {$A$};
    \node at (4.5,1.5) [name=a3] {$A$};
    \node at (6.5,1.5) {$A$};

    \draw[red!65!black, thick, ->] (a1) -- node [pos=.6, above]
    {$d_2$} (b1);
    \draw[red!65!black, thick, ->] (a2) -- node [pos=.6, above]
    {$d_2$} (b2);
    \draw[red!65!black, thick, ->] (a3) -- node [pos=.6, above]
    {$d_2$} (b3);
  \end{tikzpicture}
  \end{center}
  with 
  $$
    d_2(a) = \underbrace{a+\dots+a}_{\text{$n$ times}}.
  $$
  These spectral sequences collapse to give the familiar minimal
  resolutions
  \begin{center}
    \begin{tikzpicture}[scale=1.3]
      \node at (0,0) [name=a] {$\bbZ$};
      \node at (1.5,0) [name=b] {$\bbZ$};
      \node at (3,0) [name=c] {$\bbZ$};
      \node at (4.5,0) [name=d] {$\bbZ$};
      \node at (6,0) [name=e] {$\bbZ$};
      \node at (7.5,0) [name=f] {$\dots$};
      \draw[->] (f) -- (e);
      \draw[->] (e) -- node [midway, above] {$n\cdot$} (d);
      \draw[->] (d) -- node [midway, above] {$0$} (c);
      \draw[->] (c) -- node [midway, above] {$n\cdot$} (b);
      \draw[->] (b) -- node [midway, above] {$0$} (a);
    \end{tikzpicture}
  \end{center}
and 
  \begin{center}
    \begin{tikzpicture}[scale=1.3]
      \node at (0,0) [name=a] {$A$};
      \node at (1.5,0) [name=b] {$A$};
      \node at (3,0) [name=c] {$A$};
      \node at (4.5,0) [name=d] {$A$};
      \node at (6,0) [name=e] {$A$};
      \node at (7.5,0) [name=f] {$\dots$};
      \draw[->] (e) -- (f);
      \draw[->] (d) -- node [midway, above] {$n\cdot$} (e);
      \draw[->] (c) -- node [midway, above] {$0$} (d);
      \draw[->] (b) -- node [midway, above] {$n\cdot$} (c);
      \draw[->] (a) -- node [midway, above] {$0$} (b);
    \end{tikzpicture}
  \end{center}
We claim that we have a communting diagram
\begin{center}
  \begin{tikzpicture}
    \node at (-1,0) [name=z] {$\bbZ$};
    \node at (2,0) [name=a] {$H^4(\bbC P^\infty;\bbZ)$};
    \node at (6,0) [name=c] {$H^3_{gp}(U(1);U(1))$};
    \node at (-1,-2) [name=z2] {$\bbZ/n\bbZ$};
    \node at (2,-2) [name=a2] {$H^4(L_n^\infty;\bbZ)$};
    \node at (6,-2) [name=c2] {$H^3(\boldsymbol\mu_n;U(1))$};
    \node at (6,-4) [name=c3] {$Hom(\boldsymbol\mu_n,U(1))$};
    \node at (-1,-4) [name=z3] {$\boldsymbol\mu_n$};
    \node at (10.5,-2) [name=d2]
      {$H^3(\boldsymbol\mu_n;\boldsymbol\mu_n)$};
    \node at (10.5,-4) [name=d3]
      {$Hom(\boldsymbol\mu_n;\boldsymbol\mu_n)$,}; 

    \draw[double equal sign distance] (z) -- (a);
    \draw[double equal sign distance] (z2) -- (a2);

    \draw[double equal sign distance] (c) -- node [midway, above] {$\sim$} (a);
    \draw[double equal sign distance] (c2) -- node [midway, above]
    {$\sim$} (a2);

    \draw[->] (z) -- node [midway, left] {$q$} (z2);
    \draw[->] (a) -- node [midway, right] {$Bi^*$} (a2);
    \draw[->] (d2) -- node [midway, above] {$\cong$} (c2);
    \draw[->] (c2) -- node [midway, right] {$\cong$} (c3);
    \draw[->] (d2) -- node [midway, right] {$\cong$} (d3);
    \draw[double equal sign distance] (c3) -- (d3);
    \draw[->] (c) -- node [midway, right] {$i^*$} (c2);
    \draw[->] (z2) -- node [midway, left] {$\cong$} (z3);
    \draw[double equal sign distance] (c3) -- (z3);
  \end{tikzpicture}
\end{center}
where $q$ is the quotient map, and the equal signs refer to the
standard identifications.
The commutativite of the top left square follows from the fact that
$Bi^*$ is the edge homomorphism in our cohomology spectral sequence.
The commutativity of the bottom left square is a diagram
chase, involving the minimal resolutions for cohomology with
coefficients in $\bbZ$, $\bbR$ and $U(1)$.
It follows that the restriction of $\mU(1)^-$ to $\mn$ is
a choice of $\mC_n[i]$. This categorical group $\mC_n[i]$ with center
$U(1)$ associated to $\mC_n$ determines $\mC_n$ up to equivalence. It
follows that the categorical group associated to $\boldsymbol\kappa$
is a choice of $\mC_n$.
\end{proof}
There is an alternative description of the
categorical group $\mC_n$.
For $a\in\frac1n\bbZ$, we write
$$
  a = \lfloor a\rfloor + a',
$$
where the Gau{\ss}
bracket $\lfloor a\rfloor$ denotes the largest
integer less than or equal to $a$.
\begin{Def}[{\cite{HuangLiuYe14},
  \cite[Sec.3, p.49]{JoyalStreet93}}]
  Let $\mC'_n$ be the skelettal 2-group constructed from the
  $\boldsymbol\mu_n$-valued 3-cocycle
  $$
    \alpha(a|b|c)\;:=\; \exp(\lfloor a'+b'\rfloor c') \; =  \;  \exp(\lfloor
    a'+b'\rfloor c)
  $$
  on $\frac1n\bbZ/\bbZ$.
\end{Def}
\begin{Lem}
  We have an equivalence of 2-groups between $\mC_n'$ and $\mC_n$.
\end{Lem}
\begin{proof}
  We define a monoidal equivalence from $\mC_n'$ to $\mC_n$.
  On objects, we let $F$ be the map
  \teqnarray{F}{\frac1n\bbZ/\bbZ}{\frac1n\bbZ}{[a]}{a',}
  and on arrows, we let $F$ be the map
  \teqnarray{F}{\left(\frac1n\bbZ/\bbZ\right)\times\boldsymbol\mu_n}
    {\frac1n\bbZ\ltimes(\bbZ\times\boldsymbol\mu_n)}{([a],z)}{(a',0,z).}
  We then define the natural transformation
  \tlongmap{\phi}{F([a])+F([b])}
    {F([a]+[b])}
  given by the arrows
  $$(a'+b',-\lfloor a'+b'\rfloor,1)$$
  in $\mC_n$. It is elementary to check that $(F,\phi)$ is indeed a
  monoidal equivalence.
\end{proof}
\section{Platonic 2-groups}\label{sec:Platonic}
The previous
section will serve as blueprint for our discussion of the Platonic
2-groups. Let $G$ be a finite subgroup of the three sphere. 
It is well known\footnote{Periodicity is a theorem by Artin and Tate
  \cite{ArtinTate68}, 
  the full statement is a combination of
\cite[XII.2(4),XII.11.1, XVI.9 Application 4]{CartanEilenberg99}. 
See also \cite[Cor. 3.1]{FengHananyHePrezas04} for a direct
proof (following Schur) that the Schur multiplier vanishes 
and \cite{TomodaZvengrowski08} for an explicit resolution and a
description of the product structure in cohomology.}
that the (co)homology of $G$ is periodic with period $4$,
with the reduced integral homology concentrated in odd
degrees, 
$$%
  H_i(G)\,=\,
  \begin{cases}
    \bbZ & \text{if } \,i=0,\\
    G^{ab} & \text{if } \,i\equiv 1\mod 4,\\
    \mg & \text{if } \,i\equiv 3\mod 4,\text{ and}\\
    0 & \text{if $i>0$ is even},
  \end{cases}
$$
and the integral cohomology concentrated in even degrees,
$$%
  H^i(G)\,=\,
  \begin{cases}
    \bbZ & \text{if } \,i=0,\\
    G^{ab} & \text{if } \,i\equiv 2\mod 4,\\
    \mg & \text{if } \,i\text{ is a positive multiple of }4,\text{ and}\\
     0 & \text{else.}
  \end{cases}
$$
In particular, $G$ possesses a weak categorical Schur cover with
center $\mg$. The following proposition shows that $\mG_{uni}$ can be realized as a
sub-categorical group of the third String 2-group.
\begin{Prop}
  The restriction of $String(3)$ to $G$ is equivalent to the
  categorical group with center $U(1)$ associated to $\mG_{uni}$ via
  the canonical inclusion $i\negmedspace:\mg\hookrightarrow U(1)$,
  $$
    String(3)\arrowvert_G\,\simeq\,\mG_{uni}[i].
  $$
\end{Prop}
\begin{proof}
We follow the argument in the proof of Lemma \ref{lem:cyclic_Schur}, with the  
difference that the circle
group of complex units, $U(1)\subset\bbC$, is replaced by the three
sphere of unit quaternions, $\bbS^3\subset\bbH$.  
Viewing
$$\bbS^\infty = {colim} \:\bbS^{4n-1}$$ as the colimit of the
  spheres in $\bbH^n$, we have
  $$%
  \begin{array}{rcccl}
    B\bbS^3& \simeq& \bbS^\infty\,/\,\bbS^3 & =& \bbH P^\infty,\quad\quad\text{and}\\[5pt]
    BG&  \simeq  &  \bbS^\infty\,/\,G. &&
  \end{array}
  $$
  If $j$ is the inclusion of $G$ in $\bbS^3$, then $Bj$ becomes
  the fibration 
  \begin{center}
    \begin{tikzpicture}
      \node at (2,0) [name=A] {$\bbS^3/G$};
      \node at (4.5,0) [name=B] {$BG$};
      \draw [right hook->] (A) -- node [midway, above] {} (B);
      \node at (4.5,-2) [name=D] {$\bbH P^\infty$.};
      \draw [->>] (B) -- node [midway,right] {$Bj$} (D);
    \end{tikzpicture}
  \end{center}
  The fibre is the spherical three manifold
  $\bbS^3/G$ and, in particular,
  connected and oriented.
  We get the following picture of
  the Leray-Serre spectral sequence for integral homology.
  \begin{center}
  \begin{tikzpicture}[scale=.8]

    \draw[very thin,color=gray!20](-.6,-.6)grid(10.2,5.2);
    \draw[->, very thick](-.6,0)--(10.2,0) node[right] {$\bbH P^\infty$};
    \draw[->, very thick](0,-.6)--(0,5.2) node[above]
    {$\bbS^3\,/\,G$};

    \node at (.5,.5) {$\bbZ$};
    \node at (4.5,.5) [name=b1] {$\bbZ$};
    \node at (8.5,.5) [name=b3] {$\bbZ$};

    \node at (.5,1.5) {$\,\,G^{ab}$};
    \node at (4.5,1.5)  {$\,\,G^{ab}$};
    \node at (8.5,1.5)  {$\,\,G^{ab}$};

    \node at (.5,2.5) {};
    \node at (4.5,2.5) {};
    \node at (8.5,2.5) {};

    \node at (.5,3.5) [name=a1] {$\bbZ$};
    \node at (4.5,3.5) [name=a3] {$\bbZ$};
    \node at (8.5,3.5) {$\bbZ$};

    \draw[red!65!black, thick, <-] (a1) -- node [pos=.51, above]
    {$\,\,n\cdot$} (b1);
    \draw[red!65!black, thick, <-] (a3) -- node [pos=.51, above]
    {$\,\,n\cdot$} (b3);
  \end{tikzpicture}
  \end{center}
  The only non-trivial differential $d_4$ is multiplication by $n=|G|$.
  The remainder of the proof is identical to that of Lemma \ref{lem:cyclic_Schur}
\end{proof}
\begin{Rem}
  The spherical three manifolds turning up as fibres in the above
  proof have been the object of intense study. For instance,
  if $G$ is the binary icosahedral group, then the space
  $\bbS^3/G$ is 
  the exotic homology 3-sphere of Poincar\'e.
\end{Rem}
\begin{Rem}
  In the abelian case, where
  $G$ is a finite cyclic subgroup of $\bbS^3$,
  the inclusion $j$ factors through a maximal torus,
\begin{center}
  \begin{tikzpicture}
    \node at (0,0) [name=A] {$G$};
    \node at (2,.8) [name=B] {$\bbS^1$};
    \node at (4,0) [name=C] {$\bbS^3$,};
    \draw[right hook->] ([yshift=-1pt]A.east) -- node [midway, below] {$j$}
      ([yshift=-1pt]C.west);

    \draw[right hook->] ([yshift=-1pt]A.north east) -- node [midway, above] {$i$}
      ([yshift=-1pt]B.west);

    \draw[right hook->] ([yshift=-1pt]B.east) -- node [midway, above] {$k$}
      ([yshift=-1pt]C.north west);
  \end{tikzpicture}
\end{center}
and the commuting diagram
\begin{center}
  \begin{tikzpicture}
    \node at (0,1.4) [name=X] {$H^3_{gp}\left(\bbS^3;U(1)\right)$};
    \node at (4.5,1.4) [name=Y] {$H^3_{gp}\left(\bbS^1;U(1)\right)$};
    \draw[->] (X) -- node [midway, above] {$k^*$} (Y);

    \node at (0,-0.5) [name=A] {$H^*\left(\bbH P^\infty;\bbZ\right)$};
    \node at (4.5,-0.5) [name=B] {$H^*\left(\bbC P^\infty;\bbZ\right)$};
    \draw [->] (A) -- node [midway, above] {$(Bk)^*$} (B);
    \node at (-3,-2.4) {$|v|=4$};
    \node at (7.5,-2.4) {$|x|=2$};
    \node at (0,-2.4) [name=C] {$\bbZ[\![v]\!]$};
    \node at (4.5,-2.4) [name=D] {$\bbZ[\![x]\!]$};
    \draw [->] (C) -- (D);
    \draw [double equal sign distance] (A) -- (C);
    \draw [double equal sign distance] (B) -- (D);
    \node at (0,-3.2) [name=E] {$v$};
    \node at (4.5,-3.2) [name=F] {$x^2$};
    \draw [|->] (E) -- (F);
    \draw [double equal sign distance] (A) -- (X);
    \draw [double equal sign distance] (B) -- (Y);
  \end{tikzpicture}
\end{center}
identifies the
restriction of $String(3)$ to $\bbS^1$ with the categorical
group $\mU(1)^-$ of the previous section. 
\end{Rem}

\section{The string covers of the alternating groups}
\label{sec:Whitehead}
Let $S_n$ be the symmetric group on $n$ elements, and let $\varrho_n$ be
its permutation representation. The alternating
group $A_n\subset S_n$ is the subgroup of even permutations. We will write
$\widetilde A_n$ for its spin double cover.
In this section, we will introduce a family of categorical groups
$\mathscr A_n$, fitting into commuting diagrams
\begin{center}
  \begin{tikzpicture}

    \node at (-1.5,0) [name=A,green!50!black] {$S_n$};
    \node at (-1.5,1.5) [name=A2,green!50!black] {$A_n$};
    \node at (-1.5,3) [name=A3,green!50!black] {$\widetilde A_n$};
    \node at (-1.5,4.5) [name=A4,green!50!black] {$\mathscr A_n$};
    \node at (-1.5,6) [name=A5,green!50!black] {$1/\!\!/\pi_3(\bbS^0)$};

    \node at (2.5,0) [name=B,blue!70!black] {$O(n)$.};
    \node at (2.5,1.5) [name=B2,blue!70!black] {$SO(n)$};
    \node at (2.5,3) [name=B3,blue!70!black] {$Spin(n)$};
    \node at (2.5,4.5) [name=B4,blue!70!black] {$String(n)$};
    \node at (2.5,6) [name=B5,blue!70!black] {$1/\!\!/U(1)$};

    \draw[>->,green!50!black] (A5) -- (A4);
    \draw[>->,blue!70!black] (B5) -- (B4);

    \draw[right hook->] (A) -- node [midway, above] {$\varrho_n$} (B);
    \draw[right hook->] (A2) -- node [midway, above] {$\bar\varrho_n$} (B2);
    \draw[right hook->] (A3) -- node [midway, above] {$\widetilde\varrho_n$} (B3);
    \draw[right hook->] (A4) -- (B4);
    \draw[right hook->] (A5) -- node [midway, above] {$e$} (B5);

    \draw[->>,green!50!black] (A3) -- node [midway, left]
    {$\sigma_n$\,} (A2);
    \draw[->>,green!50!black] (A4) -- (A3);
    \draw[right hook->,green!50!black] (A2) -- node [midway, left] {$\iota_n$\,} (A);

    \draw[->>,blue!70!black] (B3) -- node [midway, right]
    {$\,\kappa_n$} (B2);
    \draw[->>,blue!70!black] (B4) -- (B3);
    \draw[right hook->,blue!70!black] (B2) -- node [midway, right]
    {$\,\varepsilon_n$} (B);
  \end{tikzpicture}
\end{center}
Here $e$ is the Adams $e$-invariant, the arrows with Greek names are
the canonical maps, and in each tower, the top two 
vertical arrows describe a categorical central extension.
These $\mathscr A_n$ are characterized, up to equivalence, by
$$
  \mathscr A_n[e] \,\simeq\, String(n)\arrowvert_{\widetilde A_n}.
$$
\begin{Def}
  We will refer to categorical groups $\mathscr A_n$ as above as the
  {\em string covers of the 
  alternating groups} or simply as the {\em alternating 2-groups}.  
\end{Def}
\subsection{The Whitehead tower of the plus construction}
  The content of this section is folclore, see for instance the
  {\em Mathoverflow} discussion {\em Plus construction considerations}.
  Let $X$ be a connected CW-complex with basepoint, 
  whose fundamental group
  has perfect commutator subgroup $$P=[\pi_1(X),\pi_1(X)].$$
  Let 
  \tlongmap pX{X^+}be a homology isomorphism
  such that 
  \begin{equation}
    \label{eq:plus}
    p_*(P)\,=\,0\,\subseteq\,\pi_1\left(X^+\right).
  \end{equation}
  These conditions are satisfied if and only if the map $p$ is
  universal, in the homotopy category, with respect to the property \eqref{eq:plus}.
  This universal property of the plus construction determines $X^+$ up
  to unique isomorphism in the homotopy category. 
  We use the notation $X^+$ whenever
  the above conditions are satisfied, even when we are working in the
  strict category.
  Given $p$ as above, we may pull back the Whitehead tower of $X^+$ to a tower of
  fibrations over $X$,
  \begin{center}
    \begin{tikzpicture}
      \node at (1.8,0) [name=B] {$X^+$};    
      \node at (3,0) [name=A] {$W_1$};    
      \node at (5.5,0) [name=W1] {$W_2$};    
      \node at (8,0) [name=W2] {$W_3$};
      \node at (10,0) [name=D] {$\dots$};
      \draw[->>] (W1) -- (A);
      \draw[->>] (W2) -- (W1);
      \draw[->>] (D) -- (W2);    

      \node at (1.8,1.8) [name=B2] {$X$};    
      \node at (3,1.8) [name=A2] {$X_1$};    
      \node at (5.5,1.8) [name=W12] {$X_2$};    
      \node at (8,1.8) [name=W22] {$X_3$};
      \node at (10,1.8) [name=D2] {$\dots$};
      \draw[->>] (W12) -- node [above] {$\xi_1$} (A2);
      \draw[->>] (D2) -- (W22);
      \draw[->>] (W22) -- node [above] {$\xi_2$} (W12);
      \draw[->] (B2) -- node [midway, left] {$p$} (B);
      \draw[->] (W12) -- node [midway, right] {$p_2$} (W1);
      \draw[->] (A2) -- node [midway, right] {$p_1$} (A);
      \draw[->] (W22) -- node [midway, right] {$p_3$} (W2);
      \draw[double equal sign distance] (A) -- (B);
      \draw[double equal sign distance] (A2) -- (B2);
    \end{tikzpicture}
  \end{center}
  Using the Lerray-Serre spectral sequence, one shows inductively that
  the $p_i$ are homology isomorphisms. So, the fundamental
  group of $X_i$ is perfect for $i>1$, and
  \tlongmap{p_i}{X_i}{W_i}satisfies the universal property for its plus construction.
  In particular, $W_i$ is a choice for $X_i^+$, and
  the homology of the tower $X_\bullet$ encodes the homotopy groups of the plus
  construction of $X$. More precisely, 
  $$%
    \widetilde H_i(X_j) \,=\,
    \begin{cases}
      0& \quad\text{if } i< j,\\
      \pi_i(X^+) &\quad \text{if } i=j.
    \end{cases}
  $$
  It is possible to construct the tower $X_\bullet$
  directly from $X$. For this, we let $X_1=X$ and 
  then proceed inductively, as follows: once $X_i$
  has been constructed, let $K(H_{i}(X_i),i)$ be the $i$th
  Eilenberg-MacLane space for the group $H_i(X_i)$, and
  define $\xi_i$ as the fibration classified by the map
  \tlongmap{f_i}{X_{i}}{K(H_{i}(X_i),i)}corresponding to
  $id_{H_i(X_i)}$ under the universal coefficient theorem. 
  We will refer to the resulting tower as the {\em homology tower} of
  $X$. 
  For instance, $X_2=\widetilde X/P$ is the quotient of the universal
  cover of $X$ by the perfect group $P$.
\begin{Thm}\label{thm:BPQ}
  We have a diagram of pull-back squares,
\begin{center}
  \begin{tikzpicture}
    \node at (-5.5,1) [name=Y2] {$BS_{n}$};
    \node at (-5.5,3) [name=Y3] {$BA_{n}$};
    \node at (-5.5,5) [name=Y4] {$B\widetilde A_{n}$};

    \node at (-3,1) [name=Z2] {$BS_{\infty}$};
    \node at (-3,3) [name=Z3] {$BA_{\infty}$};
    \node at (-3,5) [name=Z4] {$B\widetilde A_{\infty}$};

    \node at (-0.5,-1.7) [anchor=north, text width=2.7cm, green!50!black]
      {Whitehead tower of $Q\bbS^0$};
    \node at (-0.5,-1) [name=A,green!50!black] {$Q\bbS^0$};
    \node at (-0.5,1) [name=A2,green!50!black] {$\left(BS_{\infty}\right)^+$};
    \node at (-0.5,3) [name=A3,green!50!black] {$\left(BA_{\infty}\right)^+$};
    \node at (-0.5,5) [name=A4,green!50!black] {$(B\widetilde A_{\infty})^+$};

    \node at (3,-1.7) [anchor=north, text width=3cm, blue!70!black]
      {Whitehead tower of $\bbZ\times BO$};
    \node at (2.5,-1) [name=B,blue!70!black] {$\bbZ\times BO$};
    \node at (2.5,1) [name=B2,blue!70!black] {$BO$};
    \node at (2.5,3) [name=B3,blue!70!black] {$BSO$};
    \node at (2.5,5) [name=B4,blue!70!black] {$BSpin$};

    \node at (5.5,-1) [name=C1] {$K\left(\pi_0\left(\bbS^0\right),0\right)$};
    \node at (5.5,1) [name=C2] {$K\left(\pi_1\left(\bbS^0\right),1\right)$};
    \node at (5.5,3) [name=C3] {$K\left(\pi_2\left(\bbS^0\right),2\right)$};

    \node at (-7.5,3) [name=X3] {$\{\pm1\}$};
    \node at (-7.5,5) [name=X4] {$\bbR P^\infty$};

    \draw[->, dashed] (B) -- (C1);
    \draw[->, dashed] (B2) -- (C2);
    \draw[->, dashed] (B3) -- (C3);

    \draw[right hook->, dashed] (X4) -- (Y4);
    \draw[right hook->, dashed] (X3) -- (Y3);

    \draw[->] (A) -- node [midway, above] {\small{$Q\eta$}} (B);
    \draw[->] (A2) -- (B2);
    \draw[->] (A3) -- (B3);
    \draw[->] (A4) -- (B4);

    \draw[->] (Z2) -- (A2);
    \draw[->] (Z3) -- (A3);
    \draw[->] (Z4) -- (A4);

    \draw[->] (Y2) -- (Z2);
    \draw[->] (Y3) -- (Z3);
    \draw[->] (Y4) -- (Z4);

    \draw[->>] (Y3) -- node [midway, right] {\small{$B\iota_n$}} (Y2);
    \draw[->>] (Y4) --  node [midway, right] {\small{$B\sigma_n$}}(Y3);

    \draw[->>] (Z3) -- node [midway, right] {\small{$B\iota_\infty$}} (Z2);
    \draw[->>] (Z4) -- node [midway, right] {\small{$B\sigma_\infty$}} (Z3);

    \draw[->>,green!50!black] (A3) -- node [midway, right]
      {\small{$(B\iota_\infty)^+$}} (A2); 
    \draw[->>,green!50!black] (A4) --  node [midway, right]
      {\small{$(B\sigma_\infty)^+$}}(A3); 
    \draw[right hook->, green!50!black] (A2) -- (A);

    \draw[->>,blue!70!black] (B3) -- node [midway, right] {\small{$B\varepsilon_\infty$}} (B2);
    \draw[->>,blue!70!black] (B4) --  node [midway, right] {\small{$B\kappa_\infty$}} (B3);
    \draw[right hook->,blue!70!black] (B2) -- (B);
  \end{tikzpicture}
\end{center}
  Here
  $
    Q\bbS^0\,=\,colim\;\Omega^n\,\bbS^n
  $
  is the infinite loop space of the sphere spectrum, and
  $\left(Q\bbS^0\right)_0$ the connected component of its
  basepoint, while $\eta\negmedspace:\bbS^0\longrightarrow KO$
is the unit map.
The composition of the solid horizontal arrows
gives the maps induced, respectively, by the representation
$\varrho_n$ and its lifts $\bar\varrho_n$ and $\widetilde\varrho_n$.
\end{Thm}
\begin{proof}[Proof of Theorem \ref{thm:BPQ}]
  It is well known that $\eta$ induces isomorphisms on $$\pi_0=\bbZ \quad\text{and}\quad
  \pi_1=\bbZ/2\bbZ\quad \text{and} \quad \pi_2= \bbZ/2\bbZ.$$ So, 
  the map $Q\eta$ pulls back the first three steps of the Whitehead
  tower of $\bbZ\times BO$ to the first three steps of the Whitehead tower of $Q\bbS^0$.  
  The Barratt-Quillen-Priddy theorem yields a homology isomorphism
  \tlongmap{p}{BS_\infty}{\left(Q\bbS^0\right)_0,}satisfying
  $$p_*(A_\infty)=0\quad\quad\text{and}\quad\quad (Q\eta)\circ p=B\varrho_\infty.$$ 
  So, the Whitehead tower of
  $Q\bbS^0$ is identified with the plus construction of the homology tower of
  $BS_\infty$.
  It remains to identify this homology tower in the relevant degrees.
  The first step is the
  pull-back of $B\varepsilon_\infty$ along $B\varrho_\infty$. This is
  the non-trivial double cover $B\iota_\infty$, classified by the map
  \begin{center}
    \begin{tikzpicture}
      \node at (0,0) [name=A] {$B(sgn)\negmedspace:BS_\infty$};
      \node at (4.4,0) [name=B] {$BO$};
      \node at (8,0) [name=C] {$B\{\pm1\}$.};
      \draw[->] (A) -- node [above] {$B\varrho_\infty$} (B);
      \draw[->] (B) -- node [above] {$Bdet$} (C);
    \end{tikzpicture}
  \end{center}
  Indeed, $$X_2 = \widetilde X/P = ES_\infty/A_\infty.$$
  Next, the double cover of $B\varrho_\infty$ is
  $B\bar\varrho_\infty$ and pulls back $B\kappa_\infty$ to the
  fibration $\xi_2=B\sigma_\infty$. The classifying map $f_2$
  represents the class
  $$
    [f_2]\,\in\,H^2\left(BA_\infty;H_2\left(BA_\infty\right)\right)
  $$
  classifying the Schur cover of $A_\infty$. Finally,
  the lift of $B\bar\varrho_\infty$ to $B\widetilde A_\infty$ is
  $B\widetilde\varrho_\infty$.
\end{proof}
As an immediate consequence of the theorem, we obtain half of the tower promised in the
introduction as restrictions of the short exact sequence
\begin{center}
  \begin{tikzpicture}
    \node at (0,0) [name=a] {$A_\infty$};
    \node at (2.5,0) [name=b] {$S_\infty$};
    \node at (5.2,0) [name=c] {$\pi_1(\bbS^0)$,};

    \draw[->] (a) -- (b);
    \draw[->] (b) -- (c);
  \end{tikzpicture}
\end{center}
the Schur cover of $\widetilde A_\infty$
\begin{center}
  \begin{tikzpicture}
    \node at (0.3,0) [name=a] {$\pi_2(\bbS^0)$};
    \node at (3.1,0) [name=b] {$\widetilde A_\infty$};
    \node at (5.4,0) [name=c] {$A_\infty$,};

    \draw[->] (a) -- (b);
    \draw[->] (b) -- (c);
  \end{tikzpicture}
\end{center}
and the categorical Schur cover of the superperfect group $\widetilde
A_\infty$ 
\begin{center}
  \begin{tikzpicture}
    \node at (0.5,0) [name=a] {$1/\!\!/\pi_3(\bbS^0)$};
    \node at (3.4,0) [name=b] {$\mathscr A_\infty$};
    \node at (5.6,0) [name=c] {$\widetilde A_\infty$.};

    \draw[->] (a) -- (b);
    \draw[->] (b) -- (c);
  \end{tikzpicture}
\end{center}
In other words, we can now construct the $n$th alternating 2-group as
$$
  \mathscr A_n \,:=\,\mathscr A_\infty\arrowvert_{\widetilde A_n}.
$$
Consider the homomorphisms
$$
\begin{array}{ccccccc}
  b_1\negmedspace : H_1(S_n) &\,\longrightarrow\,&H_1(S_\infty) &\,\cong\,&
                                                            \pi_1(\bbS^0) &\,\cong\,& 
                                                            \boldsymbol\mu_2 \\ [+6pt]
  b_2\negmedspace : H_2(A_n) &\,\longrightarrow\,& H_2(A_\infty) &\,\cong\,&
                                                             \pi_2(\bbS^0) &\,\cong\,& 
                                                            \boldsymbol\mu_2\\ [+6pt]
  b_3\negmedspace : H_3(\widetilde A_n) &\,\longrightarrow\,& H_3(\widetilde
                                             A_\infty) &\,\cong\,&
                                                                   \pi_3(\bbS^0) 
                                                &\,\cong\,&
                                                            \boldsymbol\mu_{24}, 
\end{array}
$$
where the middle isomorphisms are induced by the Barratt-Priddy-Quillen map. 

\begin{Lem}[{\cite[7.2.3]{Hausmann78}}]
  The map $b_1$ is an
  isomorphism for $n\geq 2$, the map $b_2$ is an isomorphism for
  $n=4,5$ or $n\geq 8$, and the map $b_3$ is an isomorphism for $n=4$,
  $n=8$ or $n\geq 11$.
\end{Lem}
\begin{proof}
  It is well known that the abelianization of $S_n$ is $\bbZ/2\bbZ$
  for $n\geq 2$. The values where $b_2$ is an isomorphism are also well
  known. This goes back to work of Schur. For $5\leq n\leq\infty$, the
  group $A_n$ is perfect. 
  In this range, we have a compatible system of isomorphisms   
  \begin{eqnarray*}
        H_2(A_n) &\,\cong \,& \pi_2(BA_n^+).
  \end{eqnarray*}
  Similarly, we have compatible isomorphisms
  \begin{eqnarray*}
        H_3(\widetilde A_n) &\,\cong \,& \pi_3(B\widetilde A_n^+),
  \end{eqnarray*}
  for $n=5$ and $8\leq n\leq\infty$. Further, when $A_n$ is perfect, the fibration
  \begin{center}
    \begin{tikzpicture}
      \node at (0,0) [name=A] {$B\{\pm1\}$};
      \node at (3,0) [name=B] {$B\widetilde A_n^+$};
      \node at (6,0) [name=C] {$BA_n^+$.};
      \draw[>->] (A) --  (B);
      \draw[->>] (B) -- node [above] {} (C);
    \end{tikzpicture}
  \end{center}
  \cite[Prop.7.1.3]{Hausmann78} 
  yields an isomorphism
  \begin{eqnarray*}
        \pi_3(B\widetilde A_n^+) &\,\cong \,& \pi_3(BA_n^+).
  \end{eqnarray*}
  Apart from the case $n=4$, which we will treat in Lemma \ref{lem:tetrahedral}, the statement of the Lemma can now be 
  read off from
  the proof of Proposition A in 
  \cite{Hausmann78}.
\end{proof}
In low degrees, we still have: 
\begin{Cor}
  When $A_n$ is perfect, then its spin extension
  $\widetilde A_n$ is classified by the homomorphism $b_2$.    
  When $\widetilde A_n$ is superperfect, then its string extension
  $\mathscr A_n$ is classified by the homomorphism $b_3$.    
\end{Cor}
\subsection{The Adams $e$-invariant}
Given a ring spectrum $E$ with unit map $\eta=\eta_E$, we may form the exact
triangle
\[
  E[-1]\,\longrightarrow\,\overline E \, \longrightarrow \, \bbS^0
  \,\stackrel\eta\longrightarrow\, E
\]
in the stable homotopy category.\footnote{This is the
  first step in the construction of the $E$-based Adams-Novikov
  spectral sequence.}
In the case $E=KO$, we have $$\pi_{4k-1}(KO)\,=\,0\quad\quad\text{ and }\quad\quad
\pi_{4k}(KO)\,=\,\bbZ.$$ 
In positive degrees, the stable homotopy groups of spheres are finite.
It follows that for $k\geq 1$, the map $\pi_{4k}(\eta)$ is zero, so
that 
we obtain a short exact sequence
  \begin{center}
    \label{eq:ses}
    \begin{tikzpicture}
      \node at (-1.5,0) [name = Z] {$0$};
      \node at (0,0) [name=A] {$\bbZ$};
      \node at (2.4,0) [name=B] {$\pi_{4k-1}(\overline{KO})$};
      \node at (5.4,0) [name=C] {$\pi_{4k-1}(\bbS^0)$};
      \node at (7.5,0) [name = D] {$0$.};
      \draw[->] (Z) -- (A);
      \draw[->] (A) -- (B);
      \draw[->] (C) -- (D);
      \draw[->] (B) -- (C);
    \end{tikzpicture}
  \end{center}
For a finite abelian group $\pi$, we further have the isomorphism
\begin{equation}
  \label{eq:ext-hom}
  Ext(\pi,\bbZ) \,\cong\, Hom(\pi,\bbQ/\bbZ)  
\end{equation}
resulting from the injective
resolution 
$$\bbZ\longrightarrow
\bbQ\longrightarrow \bbQ/\bbZ.$$
\begin{Def}\label{def:e}
  For $k\geq 1$, we let
  \tlongmap{e}{\pi_{4k-1}(\bbS^0)}{\bbQ/\bbZ}be
  the homomorphism classifying the extension
  $[\pi_{4k-1}(\overline{KO})]$ above.
\end{Def}
\begin{Lem}
  Our definition of $e$ agrees with the definition of the Adams
  $e$-invariant in \cite[(1.1)]{AtiyahSmith74} and \cite[(4.11)]{AtiyahPatodiSingerII}.
\end{Lem}
\begin{proof}
  Following the discussion of the complex $e$-invariant in
  \cite[III,16]{ConnerFloyd66}, Atiyah and Smith identify the real $e$-invariant of a
  framed manifold $M$ of dimension $4k-1$ as
  $$%
    e(M) \,= \,
    \begin{cases}
      \phantom{\frac12}\widehat A(B) & \quad\text{$k$ even,}\\[+3pt]
      \frac12\widehat A(B) & \quad\text{$k$ odd,} 
    \end{cases}
  $$
  where $B$ is any spin manifold with boundary $\partial B=M$. 
  They argue that this is a
  well-defined element of $\bbQ/\bbZ$ by the integrality result
  \cite[Cor.2(ii)]{AtiyahHirzebruch59}.    
  To understand this formulation, consider
  the maps of exact triangles
  \[
    \includegraphics[scale=1.2]{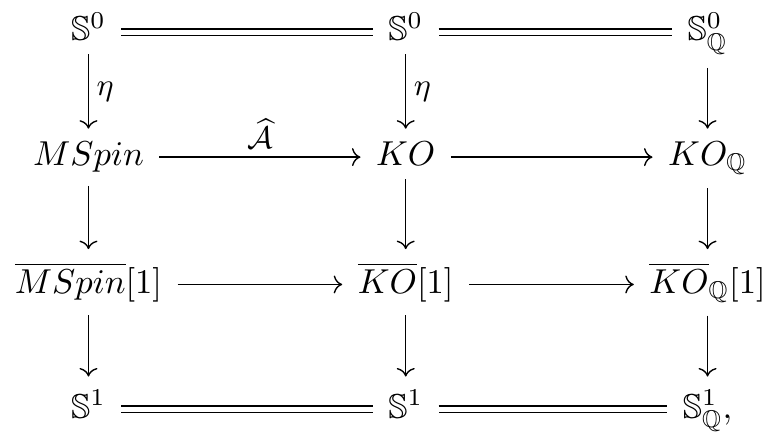}
  \]
  where $\widehat{\mathcal A}$ is the Atiyah-Bott-Shapiro orientation
  \cite{AtiyahBottShapiro64}. Following
  \cite[(7.9),(7.13),(7.17)]{LawsonMichelsohn89}, this
  yields a diagram with exact columns
  \[
    \includegraphics[scale=1.2]{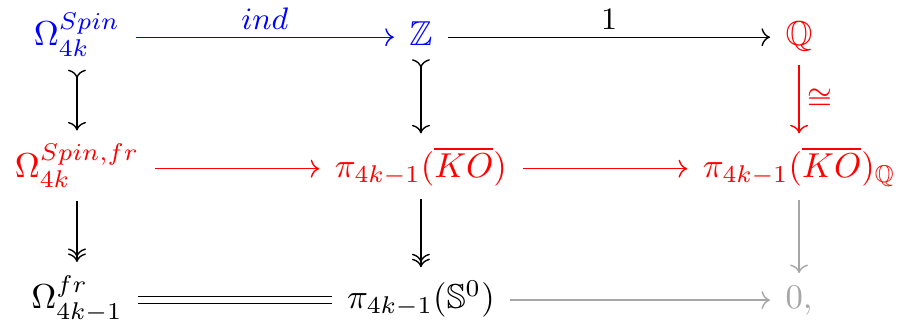}
  \]
  where $ind$ is the Atiyah-Milnor-Singer invariant,
  $$%
    ind(X) \,= \,
    \begin{cases}
      \phantom{\frac12}\widehat A(X) & \quad\text{$k$ even,}\\[+3pt]
      \frac12\widehat A(X) & \quad\text{$k$ odd.} 
    \end{cases}
  $$
  We claim that, for even $k$, the composite of the red arrows sends a spin manifold
  with framed boundary to the integral over its $\widehat
  A$-class. Indeed, this relative $\widehat A$-genus 
  is a homomorphism from
  $\Omega_{4k}^{Spin,fr}$ to $\bbQ$, which for closed manifolds agrees with
  the $\widehat A$-genus. Since the inclusion
  \[
    \includegraphics[scale=1.2]{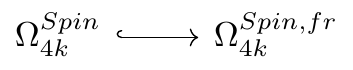}
  \]
  becomes an isomorphism after tensoring with $\bbQ$, this property
  determines the relative $\widehat A$-genus uniquely. 
  By the identical argument, the red arrows compose to half the
  relative $\widehat A$-genus for $k$ odd.
  We may now
  reformulate the definition \cite[(1.1)]{AtiyahSmith74} as follows:
  Given an element $x$ of $\pi_{4k-1}(\bbS^0)$, choose a pre-image
  $\overline x$ of $x$ in $\pi_{4k-1}(\overline{KO})$ and take $e(x)$
  to be the image
  of $\overline x$ in 
  \[
    \includegraphics{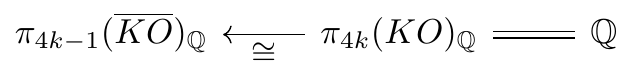}
  \]
  modulo
  \[
    \pi_{4k}({KO})\,=\, \bbZ.
  \]
  This description of $e$ coincides with the classifying map of the
  extension $[\pi_{4k-1}(\overline{KO})]$.
\end{proof}
\begin{Lem}\label{lem:hofib}
  The composite 
  \begin{center}
    \includegraphics[scale=1.2]{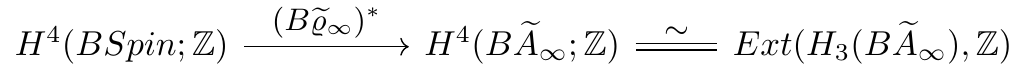}
  \end{center}
  sends the preferred generator of $H^4(BSpin;\bbZ)$ to the extension
  $[\pi_3\overline{KO}]$ of 
  $$
    H_3(B\widetilde A_\infty) \,\cong\,\pi_3(\bbS^0),
  $$
  used in Definition \ref{def:e}.
\end{Lem}
\begin{proof}
  The naturality of the universal coefficient theorem (the isomorphism
  in the lemma) allows us to replace $B\widetilde A_\infty$ with
  $B\widetilde A_\infty^+$ and $B\widetilde \varrho_\infty$ with
  $B\widetilde \varrho_\infty^{\,+}$. Let
  \tlongmap{\xi}{BSpin}{K(\bbZ,4)}represent the preferred
  generator. Then $\left(B\widetilde \varrho_\infty^{\,+}\right)^*([\xi])$ is
  represented by the composite
  $$
    \xi'\,=\,\xi\circ B\widetilde \varrho_\infty^{\,+}.
  $$
  We have a homotopy commutative diagram
  \begin{center}
  \includegraphics[scale=1.2]{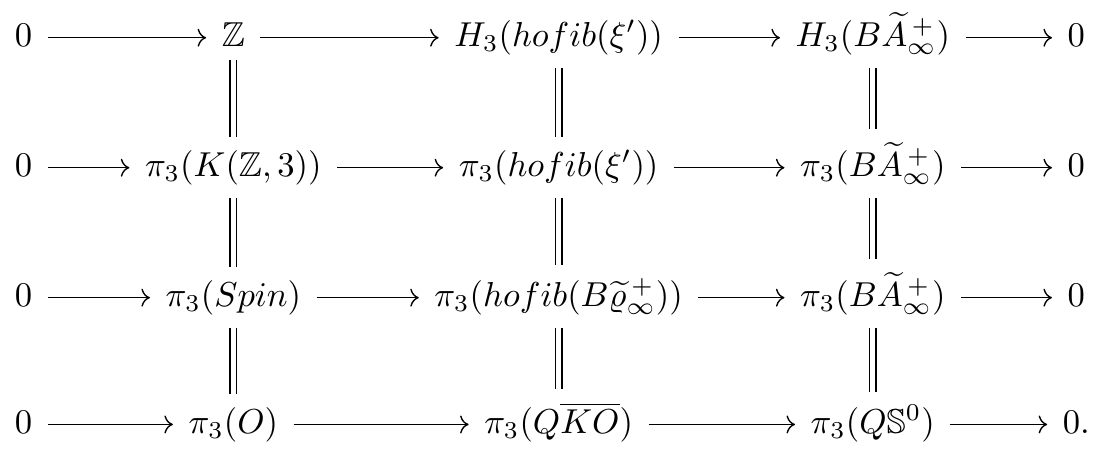}
  \end{center}
  whose rows are homotopy fiber sequences. Using the long exact
  sequence of (unstable) homotopy 
  groups, we find that all the spaces in the top two rows are
  2-connected. In fact, the second row forms the 2-connected cover of the
  third row.
  Using Hurwicz and the fact that there are no
  non-trivial homomorphisms from a finite group to $\bbZ$,
  we arrive at the following commutative diagram with
  exact rows
  \begin{center}
    \includegraphics[scale=1.2]{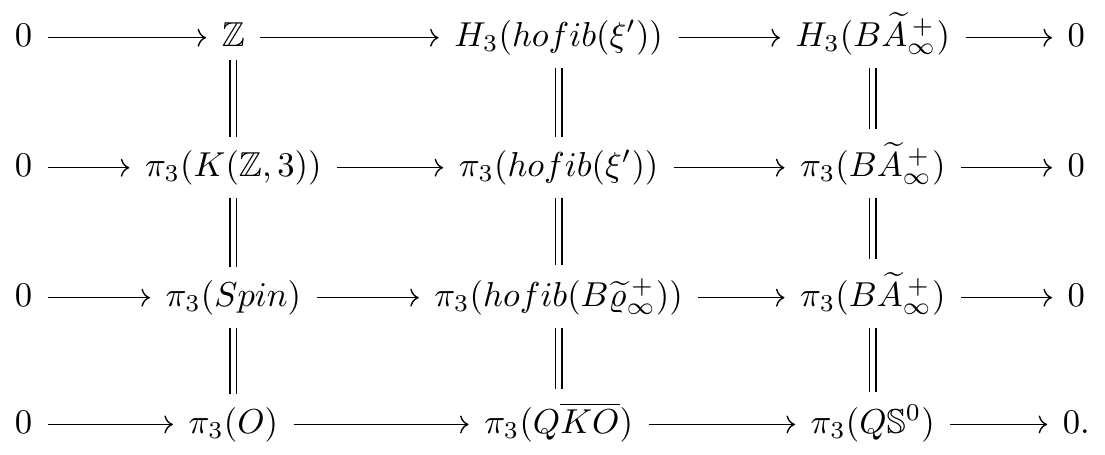}
  \end{center}
  The universal coefficient theorem identifies the class $[\xi']$ with
  the extension on the top row, while the bottom row is the extension
  in Definition \ref{def:e}. 
\end{proof}
\begin{Cor}\label{cor:A_n->String}
  The restriction of $String(n)$ to $\widetilde A_n$ is equivalent, in a manner
  unique up to unique isomorphism, to the categorical group with
  center $U(1)$ associated to $\mathscr A_n$ via the Adams $e$-invariant,
  $$\mathscr A_n[e]\,\simeq\, String(n)\arrowvert_{\widetilde A_n}.$$
\end{Cor}
\begin{proof}
This follows from the commutativity of the diagram
  \begin{center}
    \includegraphics[scale=1.2]{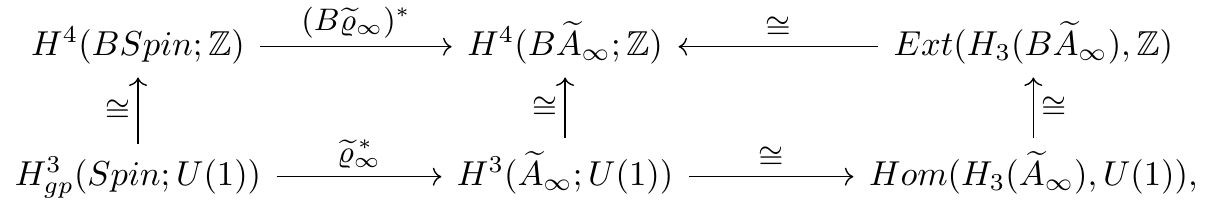}
  \end{center}  
  where the 
  horizontal isomorphisms on the right are given by the universal coefficient
  theorem, and the right-most vertical isomorphism is
  \eqref{eq:ext-hom}. 
  The left two vertical isomorphisms come from the long exact
  sequence associated to the short exact sequence of coefficients
  $$\bbZ\longrightarrow\bbQ\longrightarrow \bbQ/\bbZ.$$
\end{proof}
One interpretation of the isomorphism \eqref{eq:ext-hom} uses the fact that
the circle group $U(1)$ is a classifying space for $\bbZ$.
So, the central extensions of $\pi$ by $\bbZ$ are classified by
homotopy classes of group homomorphisms from $\pi$ to $U(1)$, and we
have
$$K(\bbZ,4)\,=\,B^3U(1).$$
Applying the construction $B(-)^+$ to the categorical central
extensions of this section, we obtain the map of exact triangles
  \[
    \includegraphics[scale=1.2]{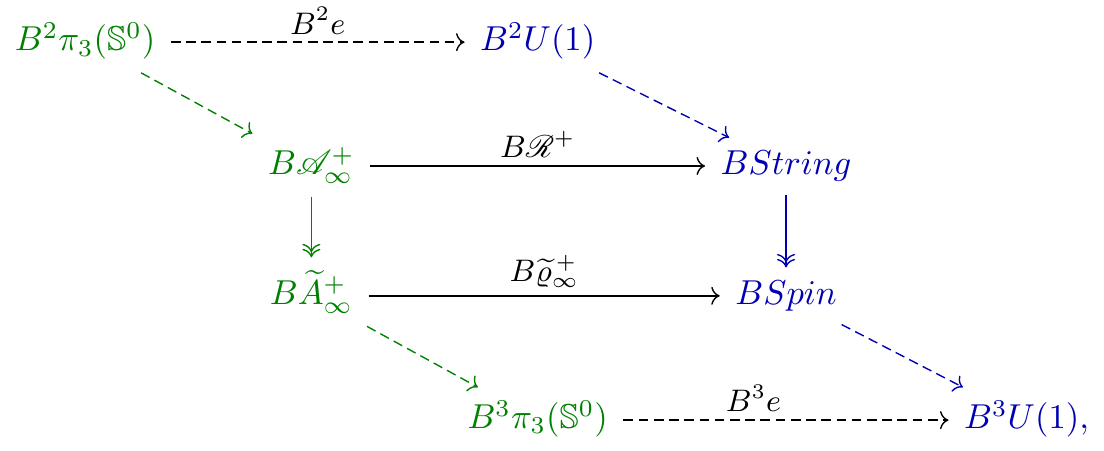}
  \]
  whose middle square adds another floor to the map of Whitehead
  towers in Theorem \ref{thm:BPQ}. 
  In particular, 
  $$
    H_4(B\mathscr A_\infty;\bbZ) \:\cong\:
    \pi_4(\bbS^0)\: = 0.
  $$
\begin{Lem}\label{lem:tetrahedral}
  The canonical inclusion of $\widetilde A_4$ in $\widetilde A_\infty$
  induces an isomorphism in degree three homology, sending the
  fundamental class of the tetrahedral spherical 3-form to the second
  Hopf map $\nu\negmedspace:{\bbS^7}\longrightarrow\bbS^4$,
  \[
    \includegraphics[scale=1.2]{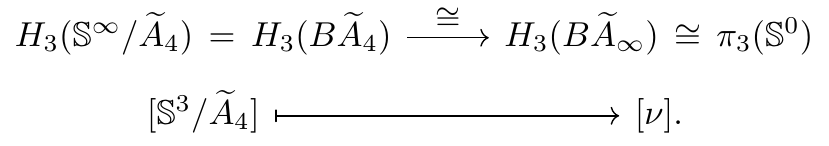}
  \]
\end{Lem}
\begin{proof}
  Applying the plus construction (with respect to $\widetilde
  A_\infty$) to the fibration
  \[
    \includegraphics[scale=1.2]{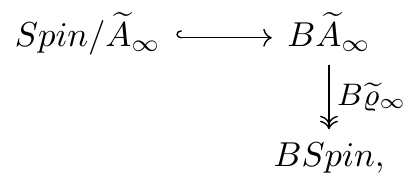}
  \]
  we obtain the identification
  \[
    \left(Spin/\widetilde
      A_\infty\right)^+\,=\,hofib\left(B\widetilde\varrho_\infty^{\,+}\right), 
  \]
  see \cite[3.D.3(2)]{DrorFarjoun96}. From
  the proof of Lemma \ref{lem:hofib}, we therefore have the short exact sequence
  \[
    \includegraphics[scale=1.2]{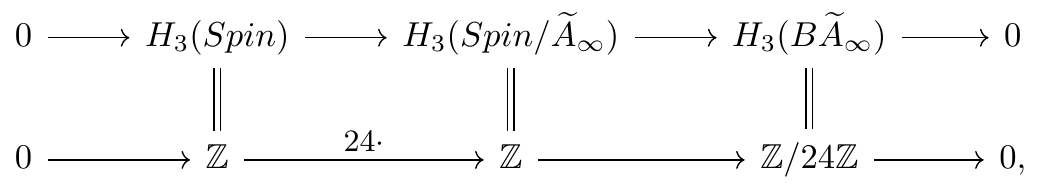}
  \]
  whose first map can be identified with the differential 
  \[
    \includegraphics[scale=1.2]{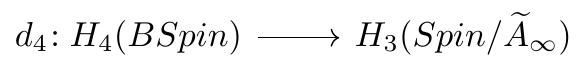}
  \]
  in the Leray-Serre spectral sequence for
  $B\widetilde\varrho_\infty$. This can be compared to the scenario
  for the platonic 2-groups. 
  In particular, we have the commuting diagram
  \[
    \includegraphics[scale=1.2]{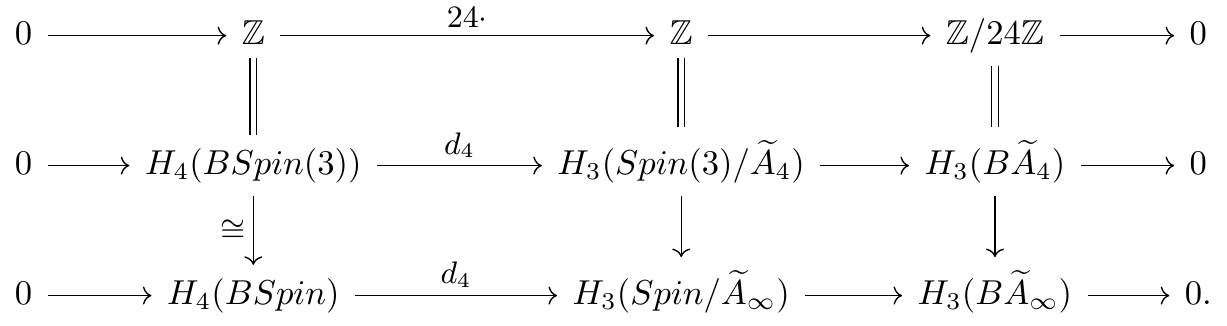}
  \]
  Here we are using a non-standard inclusion of $Spin(3)$ inside
  $Spin(4)\subset Spin$, covering the orthogonal complement of the
  trivial summand of the permutation representation. This map still
  gives an isomorphism in $H_3$, implying that all the vertival arrows
  are isomorphisms.
  It follows that the generator of $$H_3(B\widetilde
  A_\infty)\,\cong\,\pi_3(\bbS^0)$$ with $e$-invariant $\frac1{24}$ is
  the image of the fundamental class of $Spin(3)/\widetilde A_4$
  under its inclusion in $B\widetilde A_\infty$.
\end{proof}

\begin{Cor}
  The fourth alternating 2-group, $\mathscr A_4$, is the weak
  categorical Schur cover of the binary 
  tetrahedral group, while $$\mathscr A_3\,\simeq\,\mathcal C_6$$
  is the weak categorical Schur cover of the cyclic group on six elements.  
\end{Cor}


\section{Explicit constructions}
\label{sec:Explicit_constructions}
This section recalls the construction of the String 2-groups given in
\cite{Wockel11} and \cite{WagemannWockel15}. We only discuss the
restriction to our finite subgroups. 
Following \cite{BroeckertomDieck95}, we identify the maximal torus of $Spin(n)$ with
\[
  T\,=\,\bbR^{\lfloor\frac n2\rfloor}/\bbZ^{\lfloor\frac n2\rfloor}_{ev},
\]
where
\[
  \mathfrak t\,=\,\bbR^{\lfloor\frac n2\rfloor}
\]
is the Lie algebra and
\[
\Lambda\!^\vee
\,=\, \bbZ^{\lfloor\frac n2\rfloor}_{ev
}\,=\,\{m\in\bbZ^{\lfloor\frac n2\rfloor}\mid\langle m,m\rangle\}
\]
is the coweight lattice.
The basic bilinear form $\langle -,-\rangle$ on $\mathfrak{spin(n)}$ is then the multiple
of the Killing form that restricts to the standard scalar product on $\bbR^{\lfloor\frac n2\rfloor}$. 
The Cartan three form is the invariant three form $\nu$ on $Spin(n)$ with
\[
  \nu_1(\xi,\zeta,\eta)\,=\,\langle[\xi,\zeta],\eta\rangle.
\]
Restricted to $\bbS^3=Spin(3)$, we have
\[
\nu_1(\xi,\zeta,\eta)\,=\,\langle\xi\times\zeta,\eta\rangle\,=\,\det(\xi,\zeta,\eta
).
\]
So, $\nu$ is the volume form.

Let now $G\subset Spin(n)$ be a finite subgroup, and let
\[
  \bbZ\longleftarrow Bar_\bullet G
\]
be the bar resolution,  
\[
  Bar_kG\,=\,\bbZ[G]G^k.
\]
Let $C_\bullet (Spin(n))$ be the singular chain complex of $Spin(n)$. Since $Spin(n)$ is 2-connected, and $Bar_\bullet G$ is free, we may choose maps
\[
  f_i\negmedspace : Bar_iG\longrightarrow C_i(Spin(n)),
\]
for $0\leq i\leq 3$, such that $f_0$ maps $g()$ to the 0-simplex $g$ in $Spin(n)$, and the $f_i$ fit together to form a map of truncuated chain complexes of $\bbZ[G]$-modules. Here $G$ acts on the simplices in $Spin(n)$ by left translation.

Explicitly, a choice of $f$ amounts to,
for each $g\in G$, a path $\gamma_g$ from $1$ to $g$, for each
pair $(g|h)$ of elements of $G$, a 2-simplex $\Delta_{g,h}$ bounding
\[
  \gamma_g-\gamma_{gh}+g\gamma_h, 
\]
and for each triple $(g|h|k)$, a 3-simplex $W_{g,h,k}$ bounding
\[
  -\Delta_{g,h} +\Delta_{g,hk}-\Delta_{gh,k}+g\Delta_{h,k}.
\]
\begin{Def}
  For a fixed choice of $f_\bullet$, let
  \[
    \alpha\negmedspace : G^3\longrightarrow \bbR/\bbZ
  \]
  be the 3-cocycle
  \[
  \alpha(g|h|k)\,=
  \,\frac1{2\pi^2}\int_{W_{g,h,k}}\nu\quad \mod \bbZ.
  \]
\end{Def}
\begin{Lem}
  Different choices of $f_\bullet$ yield cohomologous choices of $\alpha$.
\end{Lem}
\begin{proof}
Let $f_\bullet'$ be a second choice for $f_\bullet$, and let $\alpha'$ be the resulting 3-cocycle. Employing again the 2-connectedness of $Spin(n)$, we obtain a chain homotopy
\[
  \includegraphics[scale=1.2]{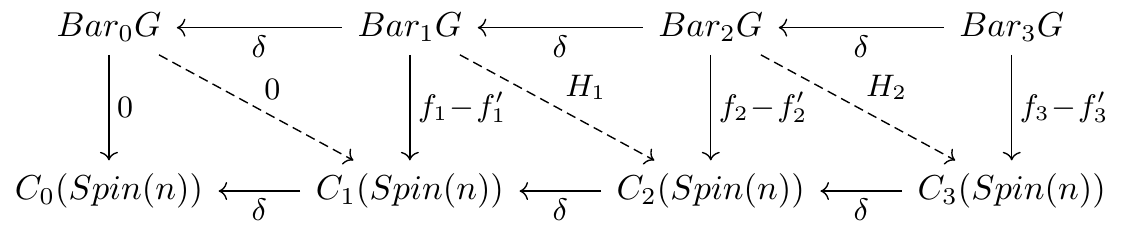}
\]
relating $f_\bullet$ and $f'_\bullet$ up to degree 2 and such that
\[
  f_3-f'_3-H_2\circ\delta
\]
takes values in the 3-cycles $Z_3(Spin(n))$.  
Letting $\beta$ be the 2-cocycle
\[
  \beta(g|h)\,=\,\frac1{2\pi^2}\int_{H_2(g|h)}\nu\quad\mod\bbZ, 
\]
it follows that
\[
  \alpha-\alpha'\,=\,\delta^*\beta.
\]  
\end{proof}
\begin{Rem}
  In \cite {FeminaGalvesNetoSreafico}, Femina, Galves, Neto and Sreafico
  describe the fundamental domain 
  of the action of $2T$ on the three sphere as
  an octahedron (the join of two geodesic segments). This yields a
  specific description of the fundamental class of $\bbS^3/2T$.
  It would be interesting to identify this class with an explicit group
  cocycle or to give a more direct relationship with the second Hopf
  map. 
\end{Rem}


\end{document}